\documentclass[10pt]{amsart}

\usepackage{amsmath}
\usepackage{amsthm}
\usepackage{amssymb}
\usepackage{mathrsfs}
\usepackage[pagebackref=true]{hyperref}
\usepackage[all]{xy}

\usepackage{diagrams}

\theoremstyle{plain}
\newtheorem{thm}{Theorem}[section]
\newtheorem{dfn}[thm]{Definition}
\newtheorem{prp}[thm]{Proposition}
\newtheorem{cor}[thm]{Corollary}
\newtheorem{lma}[thm]{Lemma}

\theoremstyle{remark}
\newtheorem{rmk}[thm]{Remark}
\newtheorem{exm}[thm]{Example}
\newtheorem{sct}[thm]{}

\def\Ee{\mathscr{E}}
\def\Pp{\mathscr{P}}
\def\Ff{\mathscr{F}}
\def\Ss{\mathscr{S}}

\def\el{\mathrm{el}}

\def\RR{\mathbb{R}}

\def\el{\mathrm{el}}

\def\Aa{\mathscr{A}}

\def\Cc{\mathscr{C}}

\def\Ob{\mathrm{Ob}}

\def\PP{\mathbb{P}}
\def\NN{\mathbb{N}}

\def\into{\rightarrowtail}
\def\onto{\twoheadrightarrow}
\def\Sh{\mathrm{Sh}}
\def\supp{\mathrm{supp}}

\def\Aut{\mathrm{Aut}}
\def\Spl{\mathrm{Spl}}
\def\BB{\mathbb{B}}
\def\Sgm{\mathfrak{S}}
\def\Ff{\mathscr{F}}

\def\hpi{\hat{\pi}}
\def\ZZ{\mathbb{Z}}

\def\Spec{\mathrm{Spec}}
\def\Et{\mathrm{Et}}
\def\Gal{\mathrm{Gal}}
\def\Sub{\mathrm{Sub}}

\def\PP{\mathbb{P}}
\def\QQ{\mathbb{Q}}

\begin{document}

\title{On the profinite fundamental group of a connected Grothendieck topos}

\author{Clemens Berger}
\author{Victor Iwaniack}

\date{May 5, 2025}



\maketitle

\begin{center}\emph{Dedicated to the memory of Marta Bunge}\end{center}

\begin{abstract}We show that finite (i.e. locally finite and decomposition-finite) objects of a connected Grothendieck topos span a Boolean pretopos with an essentially unique Galois point. The automorphism group of such a point carries a profinite topology whose classifying topos is equivalent to the given Grothendieck topos if the latter is finitely generated. This leads to an intrinsic definition of the fundamental group of any connected Grothendieck topos.\end{abstract}

\section*{Introduction}

Grothendieck \cite{SGA1} introduced Galois categories in order to define fundamental groups of algebraic varieties by means of finitary covering theory. The purpose of this article is to approach them from the perspective of topos theory hoping this might broaden up the applicability of Galois-theoretical methods even more.

The key is a convenient notion of \emph{finiteness} for Grothendieck toposes. We call an object finite if it is locally finite and decomposition-finite, the latter meaning a finite sum of connected objects. For any Grothendieck topos $\Ee$ with finite terminal object, the full subcategory $\Ee_f$ spanned by finite objects is a \emph{pretopos} with \emph{complemented subobjects}. If $\Ee$ is connected then $\Ee_f$ is a \emph{Galois category} in the sense of Grothendieck \cite{SGA1} because there is an intrinsic way of defining an exact conservative fibre functor with values in the category $\Ss_f$ of finite sets. Closing up $\Ee_f$ under sums defines an atomic Grothendieck topos $\Ee_{sf}$. The fibre functor induces a \emph{Galois point} $x_\Ee:\Ss=\Ss_{sf}\to\Ee_{sf}$, and the automorphism group $\Aut(x_\Ee)$ of this Galois point carries a uniquely determined \emph{profinite topology} such that $\Ee_{sf}\simeq\BB\Aut(x_\Ee)$.

Connected, finitely generated Grothendieck toposes are thus up to equivalence the same as classifying toposes of profinite groups. Our construction provides a functor from ``Galois-pointed'' connected Grothendieck toposes and pointed geometric morphisms to profinite groups and group homomorphisms.

There are closely related results in literature \cite{JT, M, J, D, BM, B}, at least if the existence of a point is granted, often based on compact zero-dimensional localic groups rather than on the dual profinite groups, yet our method seems more elementary than those available in literature. Barr's abstract Galois theory \cite{Barr1,Barr2} is closest to our's insofar as he also constructs a fibre functor in an intrinsic way.

In any finitely generated Grothendieck topos \emph{finite objects} coincide with \emph{coherent objects} in the sense of Grothendieck-Verdier \cite{SGA4}. Finitely generated Grothendieck toposes are thus \emph{locally coherent} and therefore have enough points by Deligne's Theorem \cite{SGA4}. In particular, by a general Representation Theorem of Butz-Moerdijk \cite{BM}, in the connected, atomic case, they are equivalent to classifying toposes of \emph{topological groups}. The main contribution of the present article consists in identifying this topological group for a connected, finitely generated Grothendieck topos $\Ee$ with the \emph{profinite} automorphism group $\Aut(x_\Ee)$ of an explicit \emph{Galois point} $x_\Ee$.

The definition of Galois point refers to the notion of Galois object. Any finite Galois object $A$ of a connected Grothendieck topos $\Ee$ determines, and is determined by, a geometric quotient $q_A:\Ee\to\BB\Aut(A)$ where $\BB\Aut(A)$ denotes the topos of $\Aut(A)$-sets for the finite automorphism group $\Aut(A)$. A Galois point $x_\Ee:\Ss\to\Ee$ is defined by the property that the composite $q_A\circ x_\Ee:\Ss\to\BB\Aut(A)$ is isomorphic to the canonical point $\Ss\to\BB\Aut(A)$ for each finite Galois object $A$ of $\Ee$.

It is remarkable that any two Galois points of a connected, finitely generated Grothendieck topos are isomorphic. The profinite fundamental group is thus independent of the chosen Galois point. In view of Moerdijk-Wraight \cite{MW}, this is analogous to the fact that the fundamental group of a path-connected topological space does not depend on the choice of base point. In the special case of the topos $\Et(k)$ of sheaves on the \emph{\'etale site} of a field $k$, choosing a separable closure $\bar{k}$ of $k$ yields a Galois point $x_{\Et(k)}:\Ss\to\Et(k)$ and the profinite fundamental group at this point is the absolute Galois group $\Gal(\bar{k}/k)$ with its Krull topology.

During lifetime, Marta Bunge has much investigated the problem of suitably defining the \emph{fundamental groupoid} of a Grothendieck topos. Her work goes far beyond of what is done here. The definition of a fundamental groupoid without assuming connectedness encompasses first and foremost an understanding of the ``topos of connected components'' of a Grothendieck topos. To this purpose Marta Bunge elaborated the concept of a multi-Galoisian point \cite{B} well-suited in a Galois-theoretical context. The hyperpure/complete spread factorisation, developed in collaboration with Jonathon Funk \cite{BF}, provides a satisfactory answer in general.\vspace{1ex}

We tried to keep the article reasonably self-contained using the terminology of the textbooks of Mac Lane-Moerdijk \cite{MM} and Johnstone \cite{J} if not otherwise stated.

In Section 1 we review decidable objects.

In Section 2 we review locally connected and locally constant objects.

In Section 3 we introduce our notion of finiteness. Proposition \ref{comparison} relates local finiteness to decidable Kuratowski-finiteness. The subcategory $\Ee_{sf}$ of sums of finite objects of any Grothendieck topos $\Ee$ with finite terminal object is shown to be an atomic Grothendieck topos, cf. Theorem \ref{main1}. Finitely generated Grothendieck toposes are shown to be locally coherent, cf. Proposition \ref{coherent}. In Theorem \ref{classified} several characterisations of connected, finitely generated Grothendieck toposes are given, e.g. they are precisely the pointed, hyperconnected and separated toposes.

Section 4 investigates connected, finitely generated Grothendieck topos by Galois-theoretical methods. Proposition \ref{finiteGalois} shows that they are generated by finite Galois objects. Galois points are introduced and the essential uniqueness of Galois points is shown in Proposition \ref{Galoispoint}. Proposition \ref{Galoissite} provides a site characterisation of connected, finitely generated Grothendieck toposes. In Proposition \ref{exact}, a fibre functor is constructed turning the embedded pretopos $\Ee_f$ into a Galois category in the sense of Grothendieck.  Theorem \ref{main2} finally shows that any connected, finitely generated Grothendieck topos is equivalent to the classifying topos of the profinite automorphism group of any of its Galois points.

In Section 5 functoriality properties of the profinite fundamental group construction are established and a few examples are discussed. In particular, we show that pointed geometric morphisms $(\Ee,x)\to(\Ff,y)$ between connected, Galois-pointed Grothendieck toposes induce pointed geometric morphisms $(\Ee_{sf},\bar{x})\to(\Ff_{sf},\bar{y})$ and thus homomorphisms of profinite fundamental groups $\hpi_1(\Ee,x)\to\hpi_1(\Ff,y)$.

If the point $x:\Ss\to\Ee$ is \emph{comprehensive} (cf. Definition \ref{comprehensive}) then $x:\Ss\to\Ee$ factors through a \emph{universal covering projection} $\bar{\Ee}\to\Ee$ and there is a \emph{Chevalley fundamental group} $\pi_1(\Ee,x)$ consisting of deck transformations of this universal covering. We show in Proposition \ref{profinitecompletion} that in this favourable case the profinite fundamental group $\hpi_1(\Ee,x)$ is indeed the profinite completion of the Chevalley fundamental group $\pi_1(\Ee,x)$. In Proposition \ref{fundamental} we provide a topos-theoretical construction of Grothendieck's short exact sequence for \emph{\'etale fundamental groups}.\vspace{1ex}

{\bf Acknowledgements.} We are grateful to Ivan Di Liberti, Martin Escard\'{o}, Simon Henry and Nima Rasekh for helpful discussions concerning topos theory. We would also like to thank the referee for numerous comments permitting us to eliminate inaccuracies and to ameliorate considerably the presentation. We are especially grateful for his insisting on the fact that functorial fundamental groups need base points.

\section{Decidable objects}

This section is a review of known properties of decidable objects, see Acu\~{n}a-Linton \cite{AL}. We first recall that in any topos the property that all subobjects are complemented amounts to the property that all objects are decidable. This is a valuable way to show that a topos is \emph{Boolean}. Recall that an object $X$ is \emph{decidable} if its diagonal $X\into X\times X$ is complemented, and that a subobject $Y\into X$ is \emph{complemented} if there exists a subobject $Y'\into X$ such that $Y\cap Y'=0$ and $Y\cup Y'=X$. In this case we write $X=Y+Y'$. A subobject $Y\into X$ is complemented if and only if its characteristic map $\chi_Y:X\to\Omega$ factors through $(\top,\bot):1+1\into\Omega$. Therefore, the object $2=1+1$ is also called the complemented subobject classifier.

\begin{lma}\label{retract}Any retract of a decidable object is complemented.\end{lma}
\begin{proof}Let $i:Y\into X$ be a subobject of a decidable object $X$ equipped with a retraction $r:X\to Y$. Then the following pullback square
$$\xymatrix{X\ar[r]^{(1_X,ir)} &X\times X\\Y\ar[u]^i\ar[r]_i&X\ar[u]_{\delta_X}}$$shows that $Y$ is complemented in $X$.\end{proof}

\begin{lma}\label{Boolean}The following four conditions on a topos are equivalent:
\begin{itemize}\item[(1)]all subobjects are complemented;\item[(2)]all objects are decidable;\item[(3)]the subobject classifier $\Omega$ is decidable;\item[(4)]the inclusion $(\top,\bot):1+1\into\Omega$ is an isomorphism.\end{itemize}\end{lma}

\begin{proof}The implications $(1)\implies(2)\implies(3)$ and $(4)\implies(1)$ are immediate. For $(3)\implies (4)$ note that $1$ is a retract of $\Omega$ and hence, if $\Omega$ is decidable, Lemma \ref{retract} shows that $1$ is complemented in $\Omega$ so that $1+1\into\Omega$ is an isomorphism.\end{proof}

\begin{lma}\label{decidable}The following properties hold in any topos:
\begin{itemize}\item[(1)]finite products and finite sums of decidable objects are decidable;
\item[(2)]any subobject of a decidable object is decidable;
\item[(3)]the image of a map between decidable objects is decidable;
\item[(4)]the equaliser of two maps between decidable objects is decidable;
\item[(5)]finite limits of decidable objects are decidable;
\item[(6)]the full subcategory spanned by decidable objects is regular and extensive; the inclusion of this subcategory preserves finite limits and finite sums.
\end{itemize}\end{lma}
\begin{proof}(1) is a direct verification. For (2) use that for any subobject $X$ of an object $Y$ the diagonal of $X$ is a pullback of the diagonal of $Y$. (3) and (4) follow from (2). (5) follows from (1)-(4), (6) from (1)-(5). Since monomorphisms in the subcategory of decidable objects are monomorphisms in the ambient topos, sums of decidable objects are disjoint and stable in the subcategory of decidable objects.\end{proof}

A geometric morphism $\phi_*:\Ee\leftrightarrows\Ff:\phi^*$ is called \emph{surjective} if $\phi^*$ is conservative, and \emph{implicative}\footnote{Johnstone \cite[Lemma 3.1]{J-1} calls implicative geometric morphisms \emph{subopen}. Lemma \ref{surjection} does not hold for general surjections, each Grothendieck topos being a quotient of a Boolean topos \cite{Barr0}.} if $\phi^*$ induces for each object $Y$ of $\Ff$ a functor between internal Heyting algebras $\phi_Y^*:\Sub_\Ff(Y)\to\Sub_\Ee(\phi^*(Y))$ which preserves \emph{implication}; in particular, $\phi_Y^*$ also preserves \emph{negation}. Any epimorphism $f:X\to X'$ in $\Ee$ induces a surjection $f_*:\Ee/X\leftrightarrows\Ee/X':f^*$ such that $f^*$ is \emph{logical} and hence implicative. The following lemma is well-known in this case, cf. \cite[Obs. 2.4]{AL} and \cite[Prop. 2]{BD2}.

\begin{lma}\label{surjection}The inverse image functor of an implicative surjection of toposes preserves and reflects decidable objects and complemented subobjects.\end{lma}
\begin{proof}Inverse image functors preserve finite intersections and unions of subobjects and hence decidable objects and complemented subobjects.

Let $X$ be a subobject of an object $Y$ in $\Ff$, and let $\phi_*:\Ee\leftrightarrows\Ff:\phi^*$ be an implicative surjection. Assume that $\phi^*(X)$ is complemented in $\phi^*(Y)$, i.e. $\phi^*(Y)=\phi_Y^*(X)+ \neg\phi_Y^*(X)$. Since $\phi_Y^*$ preserves negation we have $\phi^*_Y(Y)=\phi_Y^*(X\cup\neg X)=\phi_Y^*(X)\cup\neg\phi_Y^*(X)=\phi_Y^*(X)+\neg\phi_Y^*(X)=\phi^*_Y(X)+\phi_Y^*(\neg X)$. The conservative inverse image functor $\phi^*$ \emph{reflects} binary sums so that $Y=X+\neg X$ and $X$ is complemented in $Y$. Since $\phi^*$ preserves diagonals, $\phi^*$ reflects decidable objects as well.\end{proof}

\begin{prp}\label{decidablequotient}Let $R$ be an equivalence relation on an object $X$. The quotient $X/R$ is decidable if and only if the equivalence relation $R$ is complemented in $X\times X$.\end{prp}
\begin{proof}The kernel pair $R$ of the quotient map $q:X\to X/R$ may be identified with the equaliser of the pair $qp_1,qp_2:X\times X\to X/R$ inducing thus a pullback square$$\xymatrix{R\ar[r]\ar[d]&X/R\ar[d]\\X\times X\ar[r]&X/R\times X/R}$$in which the lower horizontal map is an epimorphism. By Lemma \ref{surjection} $X/R$ is complemented in $X/R\times X/R$ if and only if $R$ is complemented in $X\times X$.\end{proof}

\begin{rmk}We are grateful to Martin Escard\'{o} for having pointed out to us the validity of Proposition \ref{decidablequotient}. Since the quotient of a decidable object is in general not decidable, the full subcategory spanned by decidable objects does not in general have quotients. This is the reason for which additional assumptions are needed to get an exact\footnote{An exact category is a regular category such that equivalence relations are effective, i.e. kernel pairs of their quotients.} and extensive subcategory, i.e. a \emph{pretopos}. The preceding proposition shows that equivalence relations must be complemented in such a subcategory.\end{rmk}

\section{Locally connected and locally constant objects}

From here on we assume that $\Ee$ is a \emph{Grothendieck topos}, and hence endowed with a geometric morphism $\gamma_*:\Ee\leftrightarrows\Ss:\gamma^*$ where $\Ss$ denotes the category of sets.

\begin{dfn}An object $X$ is called \begin{itemize}\item \emph{connected} if $\,0$ and $X$ are the only complemented subobjects of $X$;\item \emph{locally connected} if $X$ is a sum of connected subobjects.\end{itemize}A topos $\Ee$ is called \emph{locally connected} if all objects of $\Ee$ are locally connected.\end{dfn}

\begin{lma}\label{uniqueness}Any two sum decompositions of a locally connected object into non-initial connected subobjects coincide up to permutation and isomorphism of the summands.\end{lma}
\begin{proof}It is enough to show that for a connected, non-void subobject $X$ of a sum $Y=\sum_{i\in I}Y_i$ there exists (a unique) $i\in I$ such that $X$ is a subobject of $Y_i$. Indeed, by stability of sums, we have $X=\sum_{i\in I}X\cap Y_i$. Since $X$ is connected, there is $i\in I$ such that $X\cap Y_i=X$ whence $X$ is a subobject of $Y_i$.\end{proof}

The set of connected complemented subobjects (i.e. \emph{connected components}) of $X$ will be denoted $\gamma_!(X)$.

\begin{prp}\label{locallyconnected}A Grothendieck topos $\Ee$ is locally connected if and only if the inverse image functor $\gamma^*:\Ss\to\Ee$ admits a left adjoint $\gamma_!:\Ee\to\Ss$.\end{prp}

\begin{proof}The connected components construction $\gamma_!$ extends to a functor which is left adjoint to $\gamma^*$. Conversely, assuming the existence of $\gamma_!$, we can construct for each object $X$ complemented subobjects $X_i$ by pulling back the elements $i:1\to\gamma_!(X)$ along the unit $X\to\gamma^*\gamma_!(X)$. Indeed, since every set is disjoint union of its elements, the $X_i$ are disjoint subobjects of $X$, and their union is $X$. None of the $X_i$ can be initial since $\gamma^*$ reflects initial objects. The complemented subobject $X_i$ of $X$ is thus taken by $\gamma_!$ to a non-initial complemented subobject of $\gamma_i(X)$ which is singleton. If $X_i=X'_i+ X''_i$ then one of the summands is taken to a singleton, and equals $X_i$, whence all $X_i$ are connected and non-initial, and $\Ee$ is locally connected.\end{proof}

\begin{rmk}A topological space $E$ is locally connected if and only if the topos $\Sh(E)$ of set-valued sheaves on $E$ is locally connected. A topological space which is simultaneously locally connected and totally disconnected has open singletons, and is thus discrete. In particular, compact zero-dimensional spaces are locally connected only when they are discrete, i.e. finite.

If $\Ss$ is \emph{not} the category of sets, there is a difference between \emph{essential} geometric morphisms (i.e. such that $\gamma^*:\Ss\to\Ee$ has a left adjoint $\gamma_!:\Ee\to\Ss$) and \emph{locally connected} geometric morphisms (i.e. such that $\gamma^*:\Ss\to\Ee$ has an $\Ss$-indexed left adjoint $\gamma_!:\Ee\to\Ss$). Barr-Par\'e \cite{BP} have characterised locally connected geometric morphisms $\gamma:\Ee\to\Ss$ for general $\Ss$ in much the same way as in Proposition \ref{locallyconnected} by replacing the notion of connected object with the notion of $\Ss$-molecule.

An object $X$ of $\Ee$ is called an $\Ss$-molecule if the map $\Omega_\Ss\to\gamma_*((\gamma^*\Omega_\Ss)^X)$ (adjoint to the diagonal) is an isomorphism. A geometric morphism $\gamma:\Ee\to\Ss$ is then locally connected if and only if it is implicative and each object of $\Ee$ is a sum of $\Ss$-molecules, cf. \cite[Theorem 14]{BP}. Even for Grothendieck toposes $\Ee$, considering non-initial connected objects as $\Ss$-molecules, i.e. as those objects $X$ of $\Ee$ for which $2^X$ has exactly two global sections, is a useful device.\end{rmk}

\begin{lma}\label{locallyconnected2}For any connected\footnote{A geometric morphism is called connected if the inverse image functor is fully faithful.}, locally connected geometric morphism of Grothen-dieck toposes $\phi:\Ee\to\Ff$, the inverse image functor preserves connected objects.\end{lma}

\begin{proof}Since $\phi$ is a locally connected geometric morphism, its inverse image functor preserves exponentials (cf. Johnstone \cite[Lemma C.3.3.1]{J}) so that we have $$\phi^*(2^Y)=\phi^*(2)^{\phi^*(Y)}=2^{\phi^*(Y)}.$$ Applying the global section functor of $\Ee$ we get$$\gamma^\Ee_*\phi^*(2^Y)=\gamma^\Ff_*\phi_*\phi^*(2^Y)=\gamma^\Ff_*(2^Y)$$where the last identity uses that $\phi$ is a connected geometric morphism. For any non-initial connected object $Y$ of $\Ff$, $2^{\phi^*(Y)}$ has thus exactly two global sections, i.e. $\phi^*(Y)$ is a non-initial connected object of $\Ee$.\end{proof}

\begin{lma}\label{complemented1}A complemented subobject of a locally connected object is locally connected.\end{lma}

\begin{proof}Write $X$ as a sum of connected objects $X_i$, and let $Y$ be a complemented subobject of $X$. Then $Y\cap X_i$ is complemented in $X_i$ and hence either $0$ or $X_i$. This implies that $Y$ is a sum of a subset of the $X_i'$s and hence locally connected.\end{proof}

For the following definition, recall that an object of a Grothendieck topos $\Ee$ is called \emph{constant} if it belongs to the essential image of the functor $\gamma^*:\Ss\to\Ee$.

A \emph{cover} (resp. \emph{open cover}) $(U_i)_{i\in I}$ of $1$ is a family of objects (resp. subterminal objects) of $\Ee$ such that the induced map $\sum_{i\in I}U_i\to 1$ is an epimorphism.

\begin{dfn}An object $X$ of a Grothendieck topos $\Ee$ is called \emph{locally constant} (resp. \emph{locally trivial}) if there is a cover (resp. open cover) $(U_i)_{i\in I}$ of the terminal object of $\Ee$ such that $X\times U_i$ is constant in the slice topos $\Ee/U_i$ for each $i\in I$.\end{dfn}

\begin{lma}\label{etale}For each topological space $B$, the locally constant objects of the topos $\Sh(B)$ of set-valued sheaves on $B$ correspond to covering spaces of $B$.\end{lma}

\begin{proof}The equivalence between the categories of local homeomorphisms $E\to B$ and of sheaves on $B$, restricts to an equivalence between the categories of coverings $E\to B$ and of \emph{locally trivial} sheaves on $B$. Since any object of $\Sh(B)$ is a quotient of a sum of subterminal objects, \emph{locally constant} sheaves can be trivialised by open covers of the terminal object, and are thus \emph{locally trivial}.\end{proof}

\begin{rmk}\label{Galois}A topos is called \emph{localic} if it is generated by its subterminal objects or, equivalently, by its locally trivial objects. The previous proof shows that in a localic topos, locally constant objects are the same as locally trivial objects.\end{rmk}

\begin{prp}\label{finitedecidable}Locally constant objects of a Grothendieck topos are decidable.\end{prp}
\begin{proof}Constant objects are decidable since all sets are decidable and the inverse image functor preserves decidable objects. Finite sums of decidable objects are decidable by Lemma \ref{decidable}(1). A general sum can be written as filtered colimit of finite sums. In a Grothendieck topos, filtered colimits commute with finite limits. It follows that general sums of decidable objects are decidable. Let $X$ be a locally constant object and $(U_i)_{i\in I}$ a cover of $1$ such that $X\times U_i$ is constant in $\Ee/U_i$ for each $i\in I$. Putting $U=\sum_{i\in I} U_i$, we thus get an epimorphism $U\onto 1$ such that $X\times U$ is a decidable object of $\Ee/U$. Since epimorphisms induce implicative surjections between slice categories, Lemma \ref{surjection} implies that $X$ is decidable.\end{proof}

\section{Finite objects}This section introduces a suitable notion of \emph{finiteness} for Grothendieck toposes. Our notion combines \emph{local finiteness} and \emph{decomposition-finiteness}. Surprisingly, to the best of our knowledge, the combination of both has not been studied so far, cf. \cite{SGA4,KLM,AL,H}. If terminal objects are finite, the full subcategory spanned by finite objects is an essentially small \emph{pretopos} with \emph{complemented subobjects} (cf. Proposition \ref{pretopos}). This induces a factorisation of the global section functor into a \emph{connected} followed by an \emph{atomic, separated} geometric morphism (cf. Theorem \ref{main1}) and leads to several characterisations of connected, finitely generated Grothendieck toposes (cf. Theorem \ref{classified}).

\begin{dfn}\label{finite}An object $X$ of a Grothendieck topos $\Ee$ is called\begin{itemize}
\item \emph{locally finite} if there is a cover $(U_i)_{i\in I}$ of the terminal object of $\Ee$ such that $X\times U_i\cong\gamma^*(\{1,\dots,n_i\})\times U_i$ in $\Ee/U_i$ for each $i\in I$;\item \emph{decomposition-finite} if it is a finite sum of connected objects;\item \emph{finite} if it is locally finite and decomposition-finite.\end{itemize}\end{dfn}

\begin{rmk}\label{etalebis}Set-valued sheaves on a topological space $E$ induce local homeomorphisms with values in $E$ and vice-versa, cf. Lemma \ref{etale}. Covering spaces correspond hereby to locally constant sheaves. A sheaf is finite if and only if the associated covering space has finite fibres over $E$ and finitely many connected components.\end{rmk}

\begin{lma}\label{binary}Binary sums and binary products of finite objects are finite.\end{lma}

\begin{proof}A binary sum of locally finite (resp. decomposition-finite) objects is locally finite (resp. decomposition-finite), so binary sums of finite objects are finite.

Let $X,Y$ by finite objects. Writing them respectively as finite sums of connected objects, we can decompose $X\times Y$ as a finite sum of products $X_i\times Y_j$ so that, without loss of generality, we can assume $X$ and $Y$ are both connected and locally finite. There are thus objects $U,V$ and natural numbers $m,n$ such that $X\times U\cong\gamma^*(\{1,\dots,m\})\times U$ and $Y\times V\cong\gamma^*(\{1,\dots,n\})\times V$. Thus $X\times Y\times U\times V\cong\gamma^*(\{1,\dots,m\}\times\{1,\dots,n\})\times(U\times V)$ in $\Ee/(U\times V)$, and $X\times Y$ is locally finite.

By the epi/mono factorisation system of $\Ee$ we have $X\twoheadrightarrow\supp(X)\hookrightarrow 1$ and $Y\twoheadrightarrow\supp(Y)\hookrightarrow 1$. If $X\times Y\not=0$ then $\supp(X)\times\supp(Y)\not=0$. Since $X$ and $Y$ are connected, so are their supports $\supp(X)$ and $\supp(Y)$. Moreover, by Corollary \ref{finitestability} below, they are complemented subobjects of $1$ so that (having non-void intersection) they represent the same connected component $S=\supp(X)=\supp(Y)$ of $1$.

Therefore, the binary product $X\times Y$ belongs to the connected slice topos $\Ee/S$ and has thus (according to the local sum-decomposition above) at most $mn$ connected components. In particular $X\times Y$ is decomposition-finite.\end{proof}

\begin{sct}\label{Kuratowski}{\bf Kuratowski-finiteness.} The \emph{power object} $\Pp(X)=\Omega^X$ is a \emph{monoid} with respect to \emph{join}. The least join-submonoid $\Pp_f(X)$ of $\Pp(X)$ containing the image of the \emph{singleton map} $\{\}:X\to\Omega^X$ (transpose of the characteristic map $X\times X\to\Omega$ of the diagonal) is often denoted $K(X)$ in literature. An object $X$ is called \emph{Kuratowski-finite} precisely when $\Pp_f(X)$ contains the top element of $(\Pp(X),\vee)$, cf. \cite{AL,KLM}.

An object $X$ of $\Ee$ is decidable if and only if the singleton map $\{\}:X\to\Omega^X$ factors through $2^X\hookrightarrow\Omega^X$. If $X$ is decidable, then $\Pp_f(X)\subset 2^X$ and $X$ is Kuratowski-finite if and only if $\Pp_f(X)=2^X$, cf. \cite{AL,KLM}.

In a Grothendieck topos $\Ee$, each object $X$ generates a \emph{free monoid} $X^*=\sum_{n\in\NN}X^n$, the multiplication being concatenation. There is thus a \emph{unique} monoid map $\kappa_X:(X^*,\cdot)\to(2^X,\vee)$ extending $\{\}:X\to 2^X$, and $X$ is decidable Kuratowski-finite precisely when $\kappa_X$ is an epimorphism.\end{sct}

\begin{lma}\label{complemented3}A subobject of a decidable Kuratowski-finite object is complemented if and only if it is decidable Kuratowski-finite.\end{lma}

\begin{proof}Any subobject of a decidable object is decidable by Lemma \ref{decidable}(2). We get in particular a restricted singleton map $X\to 2^X$ through which the composite $X\to Y\to 2^Y$ factors precisely when $X$ is complemented in $Y$. By the universal property of the free monoid construction we get a commutative outer rectangle
$$\xymatrix{X^*\ar[r]^{\bar{\kappa}_X}\ar@{->}[d]&\Pp_f(X)\ar[r]^{\iota_X}\ar[d]&2^X\ar[d]\\Y^*\ar[r]_{\bar{\kappa}_Y}&\Pp_f(Y)\ar[r]_{\iota_Y}& 2^Y}$$
admitting horizontal image-factorisations. The outer vertical maps admit compatible retractions inducing a retraction of the middle vertical map. Therefore $\iota_X$ is a retract of $\iota_Y$ so that $\iota_X$ is invertible whenever $\iota_Y$ is invertible, i.e. $X$ is Kuratowski-finite whenever $Y$ is Kuratowski-finite.

Conversely, if $X,Y$ are both decidable Kuratowski-finite, then $\iota_X$ and $\iota_Y$ are both invertible so that $X$, the top global element of $\Pp_f(X)$, gets identified with a global element of $2^Y$, i.e. a complemented subobject of $Y$.\end{proof}

\begin{prp}\label{comparison}In any Grothendieck topos locally finite objects coincide with decidable Kuratowski-finite objects.\end{prp}

\begin{proof}Any locally finite object $X$ is decidable by Proposition \ref{finitedecidable}. For a trivialising cover $(U_i)_{i\in I}$, we get Kuratowski-finite objects $\gamma^*(\{1,\dots,n_i\})\times U_i$ in $\Ee/U_i$. Kuratowski-finiteness of $X$ amounts to the property that $\kappa_X:X^*\to 2^X$ is an epimorphism. To be epimorphic is a local property and, as well the free monoid construction $(-)^*$, as well the complemented subobject classifier $2^{(-)}$, are preserved under $-\times U_i$. Since $\kappa_X$ is $U_i$-locally epimorphic, it is globally epimorphic.

Conversely, let $\kappa_X:X^*\to 2^X$ be epimorphic. This defines a pullback$$\xymatrix{U\ar@{->>}[r]\ar[d]&1\ar[d]^{\top}\\X^*\ar@{->>}[r]_{\kappa_X}&2^X}$$in which $U$ is a globally supported, complemented subobject of $X^*$. Let us then consider the following commutative diagram$$\xymatrix{X\times X^*\ar@{->>}[r]^{1_X\times\kappa_X}\ar[d]_{s_X}&X\times 2^X\ar[d]^{t_X}\\X^*\ar@{->>}[r]_{\kappa_X}& 2^X}$$in which the right vertical map is defined as a composite morphism$$t_X:X\times 2^X\overset{\{\}\times 1_{2^X}}{\longrightarrow} 2^X\times 2^X\overset{\vee}{\longrightarrow} 2^X$$ and the left vertical map is defined via the identifications $$s_X:X\times\sum_{n\geq 0}X^n\cong\sum_{n\geq 0}X^{n+1}\hookrightarrow\sum_{n\geq 0}X^n.$$ Since the top element $\top:1\to 2^X$ is absorbing for binary join $\vee:2^X\times 2^X\to 2^X$, the subobject $U$ of $X^*$ is stable under the map $s_X:X\times X^*\to X^*$.

The free monoid $X^*$ is graded by ``word-length'', and $s_X$ takes the piece $X^n$ of words of length $n$ to the piece $X^{n+1}$ of words of length $n+1$. Pulling back we get an analogous grading of $U$, i.e. $U=\sum_{n\in\NN}U_n$ where $U_n=X^n\times_XU$. The restricted map $(s_X)_{|U}$ takes the piece $U_n$ to the piece $U_{n+1}$.

There is a least integer $n_0\geq 0$ such that $U_{n_0}\not= 0$. We claim that $U_{n_0}$ still has global support since the restriction of $\kappa_X$ to $X^{n_0}$ is still epimorphic. Indeed, if it were not then its image would not contain the top element and $U_{n_0}$ would be void.

Now $U_{n_0}$ is the subobject of $X^{n_0}$ consisting of words in which each letter occurs exactly once, and $U_{n_0+1}$ is the subobject of $X^{n_0+1}$ consisting of those words in which the first letter occurs exactly twice but all other letters occur just once. The map $s_X$ identifies $X\times U_{n_0}$ with $U_{n_0+1}\cong\gamma^*(\{1,\dots,n_0\})\times U_{n_0}$ in $\Ee/U_{n_0}$. Since $U_{n_0}$ has global support, $X$ is locally finite.\end{proof}

\begin{cor}\label{finitestability}Complemented subobjects of finite objects are finite. The image of a morphism between locally finite objects is complemented.\end{cor}
\begin{proof}The first assertion follows by combining Lemma \ref{complemented1}, Proposition \ref{comparison} and Lemma \ref{complemented3}. For the second assertion, note first that the quotient of a locally finite object is locally finite, and hence decidable Kuratowski-finite by Proposition \ref{comparison}. We may thus conclude using again Lemma \ref{complemented3}.\end{proof}

We will say that a pretopos is \emph{embedded} if it is a full subcategory of a topos, and the inclusion functor is left exact and right exact.

\begin{prp}\label{pretopos}For any Grothendieck topos with finite terminal object, the full subcategory of finite objects is an embedded pretopos with complemented subobjects.\footnote{It would be an \emph{elementary Boolean topos} if for any finite objects $X,Y$, the internal hom $Y^X$ (so in particular $2^X$) would also be a finite object. We do not know when this is the case.}\end{prp}

\begin{proof}By Lemma \ref{complemented3} the diagonal of a decidable Kuratowski-finite object is complemented. Equalisers of parallel maps between decidable Kuratowski-finite objects are thus decidable Kuratowski-finite. Using Lemma \ref{binary}, Proposition \ref{comparison} and finiteness of the terminal object, it follows that the full subcategory $\Ee_f$ of $\Ee$ spanned by finite objects has all finite limits. Since by Corollary \ref{finitestability} the epi/mono factorisation system of $\Ee$ restricts to $\Ee_f$, the category $\Ee_f$ is exact. Finite sums of finite objects are finite as well, and are stable in $\Ee_f$ under pullback, so that $\Ee_f$ is an extensive category and hence a pretopos. The finite limits and colimits computed in $\Ee_f$ coincide with those computed in $\Ee$, i.e. $\Ee_f$ is an embedded pretopos. Lemma \ref{complemented3} implies that all subobjects in $\Ee_f$ are complemented.\end{proof}

\begin{rmk}Acu\~{n}a-Linton \cite[Theorem 1.1]{AL} establish Proposition \ref{pretopos} for the full subcategory $\Ee_{dKf}$ of decidable Kuratowski-finite objects of a topos $\Ee$. Granting their result, Proposition \ref{pretopos} follows directly from Lemma \ref{complemented1} and Proposition \ref{comparison}.

Proposition \ref{comparison} ties together Kuratowski-finiteness and local finiteness. A closely related intermediate concept are Johnstone's \emph{finite cardinals}, cf. \cite[Section D.5.2]{J}. Proposition \ref{comparison} also follows from \cite[Corollary D.5.2.6]{J} and \cite[Theorem D.5.4.13]{J}. Indeed, it is a consequence of \emph{loc. cit.} that in any elementary topos $\Ee$ with natural number object $\NN_\Ee$, an object $X$ is decidable Kuratowski-finite if and only if there exists a globally supported object $U$ such that $X\times U$ is a finite cardinal in $\Ee/U$.\end{rmk}

The key for Theorem \ref{main1} is Lemma \ref{Leroy} below, due to Leroy (cf. \cite[Lemma 2.4.2]{L}) and recalled here with proof for convenience of the reader. Leroy uses it together with a locally connected analog of Corollary \ref{finitestability} to show that for any locally connected Grothendieck topos $\Ee$, the full subcategory $\Ee_{slc}$ spanned by sums of locally constant objects is an atomic Grothendieck topos, cf. \cite[Theorem 2.4]{L}.

A Grothendieck topos is called \emph{atomic} if it is locally connected and Boolean (cf. Barr-Diaconescu \cite{BD1}). This is the case if and only if each object is a sum of atoms. An \emph{atom} is an object $A$ with precisely two subobjects, $0$ and $A$, cf. Lemma \ref{complemented1}.

A pretopos will be called \emph{atomic} if each object is a sum of atoms. In particular, every subobject is complemented and atoms coincide with non-initial connected objects. Any morphism with non-initial domain and connected codomain is thus an epimorphism. A pretopos will be called \emph{bounded} if it admits a set of generators.

\begin{lma}[Leroy]\label{Leroy}Let $\Pp$ be a bounded atomic pretopos embedded in a Grothen-dieck topos $\Ee$. Assume that $\Pp$ is stable in $\Ee$ under complemented subobject and under image. Then the full subcategory $s\Pp$ spanned by sums of objects of $\Pp$ is an atomic Grothendieck topos. The inclusion $s\Pp\hookrightarrow\Ee$ is the inverse image functor of a connected geometric morphism $\Ee\to s\Pp$.\end{lma}

\begin{proof}By Giraud's criterion it suffices to show that $s\Pp$ has finite limits, pullback-stable and disjoint sums, effective equivalence relations and a set of generators.

By definition $s\Pp$ has arbitrary sums and (using distributivity) finite products. Equalisers in $s\Pp$ can be constructed summandwise so that $s\Pp$ has finite limits, and they are computed in $s\Pp$ the same way as in $\Ee$.  Since in $\Ee$ sums are disjoint and stable they are so in $s\Pp$. For the effectiveness of equivalence relations, it suffices to show that for any equivalence relation $R\subset X\times X$ in $s\Pp$ the quotient $X/R$ (computed in $\Ee$) actually belongs to $s\Pp$. By hypothesis we can write $X$ as a sum $X=\sum_{i\in I}X_i$ of atoms of $\Pp$ so that $X\times X=\sum_{(i,j)\in I^2}X_i\times X_j$. In particular, the equivalence relation $R$ is a sum of objects $R_{ij}=R\cap(X_i\times X_j)$ of $s\Pp$.

We claim that $R_{ij}$ belongs to $\Pp$ so that in particular $R_{ii}\subset X_i\times X_i$ is an equivalence relation in $\Pp$. Indeed, in $s\Pp$ all objects are sums of decidable objects and hence decidable. By Proposition \ref{decidablequotient} the equivalence relation $R$ is thus complemented in $X\times X$. Therefore $R_{ij}$ is complemented in $X_i\times X_j$ and belongs to $\Pp$ because $\Pp$ is stable in $\Ee$ under complemented subobject. Since $\Pp$ is also stable in $\Ee$ under image, the quotient maps $q_i:X_i\to X_i/R_{ii}$ belong to $\Pp$ as well. The total quotient $X/R$ belongs then to $s\Pp$ by the following formula$$X/R=\sum_{[i]\in I/\sim}X_i/R_{ii}$$ where $i\sim j$ in $I$ precisely when $R_{ij}\not= 0$. Indeed, if $X_i\times X_j=0$ then $R_{ij}=0$ and the individual quotients $X_i/R_{ii}$ and $X_j/R_{jj}$ are disjoint subobjects of $X/R$. If $R_{ij}\not=0$ then $X_i\times X_j\not=0$, and the equaliser $R_{ij}$ of $q_1p_1,q_2p_2:X_i\times X_j\rightrightarrows X/R$ maps epimorphically to the atoms $X_i$ and $X_j$. In particular, the quotients $X_i/R_{ii}$ and $X_j/R_{jj}$ coincide in $X/R$.

Finally, a set of generators for $\Pp$ is also a set of generators for $s\Pp$, and every object of $s\Pp$ is a sum of atoms so that $s\Pp$ is an atomic Grothendieck topos.\end{proof}

\begin{thm}\label{main1}For any Grothendieck topos $\Ee$ with finite terminal object, the full subcategory $\Ee_{sf}$ of sums of finite objects is an atomic Grothendieck topos. The inclusion $\Ee_{sf}\hookrightarrow\Ee$ is the inverse image functor of a connected geometric morphism.\end{thm}

\begin{proof}Since $\Ee_f$ is an atomic pretopos by Proposition \ref{pretopos} (each object is a sum of finite connected objects, i.e. atoms) and satisfies the required stability properties by Corollary \ref{finitestability}, Leroy's Lemma \ref{Leroy} implies that $\Ee_{sf}$ is an atomic Grothendieck topos once we have shown that $\Ee_f$ is bounded. Actually, $\Ee_f$ is even essentially small. Indeed, any finite object is a finite sum of atoms, and any atom is a quotient of a generator of a given generating set $\Ee_0$ of $\Ee$. Therefore, any finite object is a quotient of a finite sum of generators of $\Ee_0$. Since in any Grothendieck topos, the quotients of an object form a set, this shows that $\Ee_f$ is essentially small.\end{proof}

\begin{dfn}\label{fingen}A Grothendieck topos $\Ee$ is said to be \emph{finitely generated} if $\Ee$ is generated by its finite objects or, equivalently, if $\Ee=\Ee_{sf}$.\end{dfn}

An object $X$ of a Grothendieck topos $\Ee$ is called \emph{compact} (resp. \emph{quasi-separated}) if any cover $(U_i)_{i\in I}$ of $X$ admits a \emph{finite} subcover (resp. if in $\Ee/X$ binary products of compact objects are compact). An object is called \emph{coherent} if it is compact and quasi-separated. A \emph{locally coherent topos} is a Grothendieck topos with a generating set of coherent objects. A \emph{coherent topos} is a locally coherent topos with coherent terminal object.\footnote{Isbell's duality \cite{Isbell} between the category of spatial frames and the category of sober spaces restricts to a duality between the category of coherent frames and the category of coherent spaces (aka spectral spaces), cf. Johnstone \cite{JStone}. The topos of sheaves on a scheme is locally coherent because the Zariski spectrum of a commutative ring is a coherent space. This explains the emphasis on coherent toposes by Grothendieck and Verdier \cite{SGA4}, and much of their terminology.} Locally coherent toposes have enough points by Deligne's Theorem \cite{SGA4}. Note that there are several equivalent ways of defining (locally) coherent toposes, cf. Johnstone \cite[Theorem D.3.3.1]{J} and Lurie \cite[Proposition C.6.4]{Lurie}.

Grothendieck and Verdier \cite[D\'efinition 2.10]{SGA4} single out an interesting subclass of coherent toposes: a \emph{Noetherian topos} is a coherent topos such that every subobject of a coherent object is coherent. Coherent objects of a Noetherian topos have stationary ascending chains of subobjects. For instance, the topos of sheaves on the Zariski spectrum of a commutative Noetherian ring is a Noetherian topos.

\begin{prp}\label{coherent}Any finitely generated Grothendieck topos is locally coherent, and its finite objects are precisely its coherent objects. If in addition the terminal object is finite then it is a coherent topos, even a Noetherian topos.\end{prp}

\begin{proof}In a finitely generated Grothendieck topos, finite objects are finite sums of atoms by Theorem \ref{main1}. Since atoms are compact, finite objects are compact. Since subobjects of finite objects are complemented, and binary products of finite objects are finite, any finite object is also quasi-separated and hence coherent. Conversely, coherent objects are finite sums of locally finite atoms, hence decomposition-finite and locally finite. The second assertion follows now from Corollary \ref{finitestability}.\end{proof}

\begin{rmk}\label{proper}For each coherent topos $\Ee$,  the full subcategory $\Ee_{coh}$ spanned by coherent objects is an essentially small \emph{pretopos}. The coherent topos $\Ee$ is equivalent to the category of sheaves on the site $(\Ee_{coh},J_{coh})$ where $J_{coh}$ is the coherent topology, cf. Grothendieck-Verdier \cite[Exercice 3.11]{SGA4}, Johnstone \cite[Theorem D.3.3.7]{J} and Lurie \cite[Theorem C.6.5]{Lurie}. This sheaf representation extends to a categorical duality between pretoposes and coherent Grothendieck toposes, cf. Makkai \cite{Ma}.\end{rmk}

We now relate finite generation to \emph{compactness} and \emph{separation}. We call a Grothendieck topos $\Ee$ \emph{compact} if its terminal object is compact. An object $X$ of $\Ee$ is thus compact precisely when the slice topos $\Ee/X$ is compact. Any family of morphisms with common codomain $X$ induces by image-factorisation a family of subobjects of $X$. The former family is epimorphic if and only if the latter family is. An object $X$ is thus compact if and only the frame of subobjects of $X$ is.

A Grothendieck topos $\Ee$ is compact if and only if the frame $\omega_\Ee=\gamma_*(\Omega_\Ee)$ of global sections of the subobject classifier is compact. This is the case if and only if $\gamma:\Ee\to\Ss$ is a \emph{proper} geometric morphism in the sense of Moerdijk-Vermeulen \cite{MV} and Johnstone \cite[Section C.3.2]{J}.

A Grothendieck topos $\Ee$ with proper diagonal $\Ee\to\Ee\times_\Ss\Ee$ is called \emph{separated}. This implies that for any two objects $X,Y$ of $\Ee$, the induced geometric morphism $\Ee/(X\times Y)\to(\Ee\times_\Ss\Ee)/(X,Y)$ is proper so that binary products of compact objects are compact, i.e. the topos $\Ee$ is \emph{quasi-separated}.

It may be difficult to check whether a Grothendieck topos $\Ee$ is separated, because the diagonal $\Ee\to\Ee\times_\Ss\Ee$ is an embedding of toposes only when $\Ee$ is localic,  cf. Johnstone \cite[Proposition B.3.3.8]{J}. For instance, a coherent topos is quasi-separated by definition, yet in general not separated without further assumptions.


A coherent topos $\Ee$ is called \emph{perfect} (cf. \cite[D\'efinition VI.2.9.1]{SGA4}) if every coherent object $X$ is \emph{finitely presentable} (i.e. $\Ee(X,-)$ preserves filtered colimits) so that $\Ee$ is locally finitely presentable. This implies that in $\Ee$ coherent objects coincide with finitely presentable objects, cf. \cite[Corollaire VI.1.2.24]{SGA4}. In particular, if a geometric morphism of perfect toposes  $\Ee\to\Ff$ is induced by a morphism of sites $\rho:\Ff_{coh}\to\Ee_{coh}$ preserving filtered colimits then the direct image functor $\Ee\to\Ff$ (i.e. precomposition with $\rho$) preserves coherent objects.

Note that any Noetherian topos is perfect, cf. \cite[Section VI.2.14]{SGA4}.

\begin{lma}\label{perfect}A perfect topos with coherent subobject classifier is separated.\end{lma}

\begin{proof}If $\Ee$ is perfect then so is $\Ee\times_\Ss\Ee$. The diagonal $\delta_*:\Ee\to\Ee\times_\Ss\Ee$ is induced by a site comorphism $\delta:\Ee_{coh}\to\Ee_{coh}\times\Ee_{coh}:X\mapsto(X,X)$ whose right adjoint $\rho:\Ee_{coh}\times\Ee_{coh}\to\Ee_{coh}:(X,Y)\mapsto X\times Y$ is a site morphism preserving filtered colimits. The direct image functor preserves thus coherent objects so that the direct image $\delta_*(\Omega_\Ee)$ of the coherent subobject classifier $\Omega_\Ee$ is a coherent frame in $\Ee\times_\Ss\Ee$. Therefore, the diagonal of $\Ee$ is proper and $\Ee$ is separated.\end{proof}

\begin{prp}\label{separated}A coherent topos with locally finite coherent objects is separated.\end{prp}

\begin{proof}Since coherent objects are decomposition-finite, locally finite coherent objects are finite objects and our category $\Ee$ is finitely generated with finite terminal object. It follows then from Theorem \ref{main1} and Proposition \ref{coherent} that $\Ee$ is Boolean and Noetherian thus perfect with $\Ee_f=\Ee_{coh}$. The subobject classifier $\Omega_\Ee=1_\Ee+1_\Ee$ (cf. Lemma \ref{Boolean}) is finite thus coherent, so that $\Ee$ is separated by Lemma \ref{perfect}.\end{proof}

\begin{thm}[cf. Corollary IV.4.8 of Moerdijk-Vermeulen \cite{MV}]\label{classified}For a \emph{connected} Grothendieck topos $\Ee$ the following five conditions are equivalent:
\begin{enumerate}
\item[(1)]$\Ee$ is finitely generated;
\item[(2)]$\Ee$ is coherent with locally finite coherent objects;
\item[(3)]$\Ee$ is coherent and separated;
\item[(4)]$\Ee$ is pointed, hyperconnected and separated;
\item[(5)]$\Ee$ is equivalent to the classifying topos of a profinite group.
\end{enumerate}\end{thm}

\begin{proof}(1)$\implies$(2). This follows from Theorem \ref{main1} and Proposition \ref{coherent}.

(2)$\implies$(3). This is Proposition \ref{separated}.

(3)$\implies$(4). Let $\Ee\to\Sh(\omega_\Ee)\to\Ss$ be the hyperconnected-localic factorisation \cite{J0} of the global section functor. Since $\Ee$ is coherent and separated, the frame $\omega_\Ee$ is coherent and regular, i.e. the frame of open subsets of a Stone space, cf. \cite{Isbell,JStone}. Therefore, connectedness of $\Ee$ implies that the underlying space is a point so that $\Ee$ is hyperconnected. By Deligne's theorem $\Ee$ has a point, cf. \cite[Section VI.9]{SGA4}.

(4)$\implies$(5). Using Joyal-Tierney's Representation Theorem \cite{JT} this is proved by Moerdijk-Vermeulen \cite[Theorem II.3.1]{MV} and Johnstone \cite[Remark C.5.3.14b]{J}.

(5)$\implies$(1). The classifying topos $\BB G$ of a profinite group $G$ is the topos of continuous $G$-sets (cf. Definition \ref{profinitegroup}) and lies as such coreflective in the presheaf topos $\Ss^G$ of all $G$-sets. The coreflection $\Ss^G\to\BB G$ is a connected geometric morphism preserving finite sums and finite objects. The finite objects of $\Ss^G$ are precisely the finite $G$-sets. Their images in $\BB G$ generate $\BB G$ since any continuous $G$-set is a sum of finite transitive $G$-sets. Therefore $\BB G$ is finitely generated.\end{proof}

\begin{rmk}\label{nonconnected}In Theorem \ref{main2} we establish implication (1)$\implies$(5) without reference to Joyal-Tierney's Representation Theorem \cite{JT} by construction of a Galois point whose profinite automophism group has the required classifying property.

The equivalence $(1)\!\!\iff\!\!(4)$ is consistent with \cite[Theorem 4.6]{H} where Henry characterises Grothendieck toposes generated by decidable Kuratowski-finite objects without assuming them to be decomposition-finite. In the connected case the difference is comparatively small, cf. Lemma \ref{locallyfinite} below.

Let us mention two classes of Grothendieck toposes which \emph{strictly} contain those appearing in Theorem \ref{classified}:

Any connected, atomic, coherent Grothendieck topos is the classifying topos of a \emph{coherent} topological group in the sense of Johnstone, cf. \cite[Theorem D.3.4.3]{J}, the latter being in general neither profinite nor even prodiscrete, see Remark \ref{BhattScholze}. An example of a non-separated atomic and coherent topos is the \emph{Myhill-Schanuel topos} of nominal sets, i.e. the classifying topos $\BB\Sgm_\infty$ of the group of bijections of a countable set, equipped with the topology of pointwise convergence. In $\BB\Sgm_\infty$, coherent and decomposition-finite objects coincide so that $\BB\Sgm_\infty$ is not Noetherian.

On the other hand, there are many connected Noetherian toposes which are not atomic. Indeed, for an irreducible commutative Noetherian ring $R$, the coherent topos $\Sh(\Spec(R))$ is atomic precisely when $R$ is a field.

One can ask how much of Theorem \ref{classified} ``survives'' \emph{without} connectedness. A complete answer is beyond the scope of this article. Let us merely \emph{conjecture} here that for a general Grothendieck topos $\Ee$ the following five conditions are equivalent:
\begin{enumerate}

\item[(1)]$\Ee$ is finitely generated;
\item[(2)]$\Ee$ is locally coherent with locally finite coherent objects;
\item[(3)]$\Ee$ is locally coherent and separated;
\item[(4)]$\Ee$ has enough points, is locally hyperconnected and separated;
\item[(5)]$\Ee$ is equivalent to the classifying topos of a profinite groupoid.

\end{enumerate}

\noindent where $\Ee$ is called \emph{locally hyperconnected} if $\Ee$ is generated by objects $A$ whose slice topos $\Ee/A$ is hyperconnected, and a groupoid $G=(G_1\rightrightarrows G_0)$ is called \emph{profinite} if it is a groupoid in Stone spaces with proper source/target map $G_1\to G_0\times G_0$.

The implications (1)$\implies$(2)$\implies$(3)$\implies$(4) and (5)$\implies$(1) are established like above using that Proposition \ref{separated} extends to locally coherent toposes. 
It is however unclear to us how to establish the missing implication (4)$\implies$(5).
\end{rmk}

\section{Galois points and Galois sites}This section is central. We set up a correspondence between connected, finitely generated Grothendieck toposes and Galois categories. This is remarkable from a historical perspective insofar as Galois categories have been defined by Grothendieck \cite{SGA1} quite a bit earlier than toposes in order to axiomatise \emph{finitary covering theory}.

We proceed in three steps. We first show that any finitely generated Grothendieck topos $\Ee$ is generated by \emph{finite Galois objects} refining the well-known generation of Galois toposes\footnote{By definition, a \emph{Galois topos} is a connected, locally connected Grothendieck topos generated by locally constant objects. Theorems \ref{main1} and \ref{classified} imply that connected, finitely generated Grothendieck toposes are precisely \emph{coherent Galois toposes} illustrating again their importance.} by Galois objects, cf. \cite{L,BD2,M,B}. We then show that for a connected Grothendieck topos $\Ee$ the embedded pretopos $\Ee_f$ admits a fibre functor $\Ee_f\to\Ss_f$ endowing $\Ee_f$ with the structure of a \emph{Galois category}. The fibre functor is finally shown to induce a \emph{Galois point} $p_\Ee:\Ss=\Ss_{sf}\to\Ee_{sf}=\Ee$ whose automorphism group $\Aut(p_\Ee)$ carries a profinite topology such that $\Ee\simeq\BB\Aut(p_\Ee)$.

This is an infinitary version of Grothendieck's Representation Theorem for Galois categories \cite[Theorem 4.1]{SGA1} because of the equivalence $\Ee_f\simeq\BB\Aut(p_\Ee)_f$. Our approach is more constructive than Theorem \ref{classified} insofar as the Galois point is made explicit by Galois-theoretical methods without referring to Deligne's Theorem.

We define \emph{Galois points} to be those which become ``canonical'' when composed with finite Galois quotients. Any two Galois points of a connected, finitely generated Grothendieck topos are isomorphic, cf. Proposition \ref{Galoispoint} so that the profinite fundamental group of a connected, finitely generated Grothendieck toposes does not depend on the choice of Galois point. Galois points are thus better behaved than general points of connected, atomic Grothendieck toposes, which have Morita-equivalent localic automorphism groups but which usually are not isomorphic.

Recall that a \emph{Galois object} is a connected, globally supported object $A$ such that the canonical right action $\rho_A:A\times\gamma^*(\Aut(A))\to A$ induces an isomorphism $(p_1,\rho_A):A\times\gamma^*(\Aut(A))\to A\times A$. An object $X$ of $\Ee$ is said to be \emph{split} by $A$ if $A\times X$ is constant in $\Ee/A$ and we denote by $\Spl(A)$ the full subcategory of $\Ee$ consisting of objects split by $A$. It is well-known that $\Spl(A)$ is equivalent to the topos $\BB\Aut(A)$ of left $\Aut(A)$-sets. The equivalence is induced by a geometric morphism $\Ee\to\BB\Aut(A)$ whose inverse image functor assigns to an $\Aut(A)$-set $M$ the object $A\times_{\gamma^*(\Aut(A))}\gamma^*(M)$ of $\Ee$, see for instance \cite[Section 2.3.6]{L}.

\begin{prp}\label{finiteGalois}Any finitely generated Grothendieck topos is generated by finite Galois objects.\end{prp}

\begin{proof}In the proof of Proposition \ref{comparison} we constructed a splitting object $U$ for any finite object $X$ of $\Ee$ as a complemented subobject of $X^n$ for convenient $n$. By Lemma \ref{binary} and Corollary \ref{finitestability} the splitting object $U$ is finite. Without loss of generality $U$ can be replaced with a connected component of $U$. The inclusion $\Spl(U)\hookrightarrow\Ee$ is then the inverse image functor $\phi^*$ of an \emph{essential} geometric morphism $\phi:\Ee\to\Spl(U)$ by \cite[Theorem 3]{BD2}. A Galois object $A$ with $\Spl(A)=\Spl(U)$ can be obtained as $\phi_!(U)$ by \cite[Theorem 4]{BD2}. Since finitely generated Grothendieck toposes are atomic by Theorem \ref{main1}, the reflection $\phi_!:\Ee\to\Spl(U)$ is an \emph{epireflection}, i.e. the unit $U\to\phi^*\phi_!(U)$ is epimorphic so that $A$ is a decidable quotient of $U$, thus a finite object by Corollary \ref{finitestability}.\end{proof}

\begin{sct}\label{Galoissection}{\bf Galois coverings and Galois points.} Since any finite object $X$ is split by the above constructed Galois object $A$, the equivalence $\Spl(A)\simeq\BB\Aut(A)$ yields a one-to-one correspondence between $\Ee$-epimorphisms $A\twoheadrightarrow X$ and $\Aut(A)$-epimorphisms $\Aut(A)\onto M$ where $M$ is a fixed $\Aut(A)$-set $M$ such that $X\cong A\times_{\gamma^*(\Aut(A))}\gamma^*(M)$. If $X$ is connected, then the corresponding $\Aut(A)$-set $M$ is transitive, i.e. an orbit. Note that $M$ is pointed by the image of the identity of $A$.

An epimorphism $A\onto X$ is said to be a \emph{Galois covering} of $X$ whenever any epimorphism $\tilde{A}\onto X$ with Galois object $\tilde{A}$ splitting $X$ factors through $A\onto X$. Any finite connected object $X$ has Galois coverings, cf. Leroy \cite[Proposition 2.4.6]{L} and Moerdijk \cite[proof of Theorem 3.2(3)$\Rightarrow$(1)]{M}. Moreover, any two Galois coverings of $X$ are isomorphic in $\Ee/X$.

Consequently, by Proposition \ref{finiteGalois}, any atom $X$ of $\Ee$ may be represented by an essentially unique pair $(A\onto X, M)$ consisting of a finite Galois covering $A\onto X$ together with a pointed, transitive $\Aut(A)$-set $M$. The pointed $\Aut(A)$-set $M$ can be realised inside $\Ee$ as the morphism-set $\Ee(A,X)$, pointed by the choice of a Galois covering $A\onto X$, with $\Aut(A)$ acting from the right on the domain. We shall denote $\Aut_X(A)$ the subgroup of $\Aut(A)$ fixing the Galois covering $A\onto X$. In particular, we get pointed $\Aut(A)$-isomorphisms $\Aut(A)/\Aut_X(A)\cong\Ee(A,X)\cong M.$

There is a more topos-theoretic description of the situation: as already mentioned, for each finite Galois object $A$, we have an equivalence $\Spl(A)\simeq\BB\Aut(A)$. The full inclusion $\Spl(A)\hookrightarrow\Ee$ is the inverse image functor of a connected geometric morphism $q_A:\Ee\to\BB\Aut(A)$. The existence of a Galois covering $A\onto X$ amounts to the existence of a ``minimal'' quotient $q_A:\Ee\to\BB\Aut(A)$ ``containing'' $X$.

It is crucial that $\BB\Aut(A)$ has a \emph{canonical} essential point $p_A:\Ss\to\BB\Aut(A)$ whose inverse image functor $(p_A)^*:\BB\Aut(A)\to\Ss$ is the \emph{forgetful functor}. The right adjoint $(p_A)_*:\Ss\to\BB\Aut(A)$ assigns to a set the same set endowed with trivial $\Aut(A)$-action, while the left adjoint $(p_A)_!:\Ss\to\BB\Aut(A)$ assigns to a set the freely generated $\Aut(A)$-set. The forgetful functor $(p_A)^*$ is thus monadic and the canonical point $p_A:\Ss\to\BB\Aut(A)$ is an essential surjection. The canonical point is up to isomorphism characterised by being essential and surjective.\end{sct}

Throughout $\Ee$ denotes a connected, finitely generated Grothendieck topos.

\begin{dfn}A \emph{Galois point} of $\,\Ee$ is a point $x_\Ee:\Ss\to\Ee$ such that there is an isomorphism of points $q_A \circ x_\Ee\cong p_A$ for each finite Galois object $A$ of $\Ee$.\end{dfn}

\begin{lma}\label{Galois0}Let $X,Y$ be atoms of $\Ee$ with Galois coverings $A_X\onto X,A_Y\onto Y$ represented by pointed, transitive $\Aut(A_X),\Aut(A_Y)$-sets $M,N$.

Any epimorphism $f:X\onto Y$ determines a pair $(\alpha_f,\beta_f)$ consisting of a group homomorphism $\alpha_f:\Aut(A_X)\onto\Aut(A_Y)$ and an equivariant surjection $\beta_f:M\onto N$ such that $(\alpha_f)_!(M)=N$. The surjection $\beta_f$ is uniquely determined by $f$ while the group homomorphism $\alpha_f$ is unique up to automorphism in $\Aut_Y(A_Y)$.\end{lma}

\begin{proof}We begin by observing that any undotted diagram
$$\xymatrix{A_X\ar@{.>>}[r]^{\bar{f}}\ar@{->>}[d]&A_Y\ar@{->>}[d]\\X\ar@{->>}[r]_f&Y}$$
admits a lifting $\bar{f}:A_X\to A_Y$ rendering commutative the diagram because $A_Y\onto Y$ is a Galois covering. Two such liftings differ by an automorphism of $A_Y$ because $A_Y$ is a Galois object. This automorphism actually belongs to $\Aut_Y(A_Y)$ because $\bar{f}$ is epic and the diagram commutes.

For any $\alpha\in\Aut(A_X)$ there is a unique $\beta\in\Aut(A_Y)$ such that $\bar{f}\alpha=\beta\bar{f}$ so that we get a group homomorphism $\alpha_f:\Aut(A_X)\onto\Aut(A_Y)$. The commutativity of the diagram implies that $\alpha_f$ takes $\Aut_X(A)$ into $\Aut_Y(B)$ yielding thus an equivariant mapping $\beta_f:M\onto N$. The latter is surjective because $M$ and $N$ are transitive. By construction the pair $(\alpha_f,\beta_f)$ induces the given epimorphism $f:X\onto Y$ which amounts to the identity $(\alpha_f)_!(M)=N$.\end{proof}

\begin{prp}\label{Galoispoint}Any two Galois points of a connected, finitely generated Grothen-dieck topos are isomorphic.\end{prp}

\begin{proof}For any Galois point $x_\Ee:\Ss\to\Ee$ we fix an isomorphism of fibre functors $(p_A)^*\cong(q_A\circ x_\Ee)^*$. This yields a natural transformation $Id_\Ss\to (x_\Ee)^*(q_A)^*(p_A)_!$. We denote its image of a singleton by $x_A\in x_\Ee^*(q_A)^*(\Aut(A))=(x_\Ee)^*(A)$.

Since $A$ is a Galois object we get a pointed $\Aut(A)$-torsor $(x_\Ee^*(A),x_A)$ in sets. Any two pointed $\Aut(A)$-torsors are canonically $\Aut(A)$-isomorphic so that for any two Galois points $x^*_\Ee,\tilde{x}^*_\Ee$ we get a canonical bijection $x^*_\Ee(A)\cong \tilde{x}^*_\Ee(A)$.

For $f:A\to B$ in $\Ee$ there is a unique $\gamma_p(f)\in\Aut(B)$ such that $$\gamma_p(f)(p^*_\Ee(f)(x_A))=x_B.$$ For composable morphisms $A\overset{f}{\longrightarrow}B\overset{g}{\longrightarrow}C$ we have $\alpha_g(\gamma_p(f))=\gamma_p(gf)$ where the group epimorphism $\alpha_g:\Aut(B)\to\Aut(C)$ has been constructed in Lemma \ref{Galois0}. Since for distinct Galois points $x_\Ee,\tilde{x}_\Ee$ we have $\gamma_p(f)=\gamma_{\tilde{p}}(f)$, the above constructed bijection $x^*_\Ee(A)\cong\tilde{x}^*_\Ee(A)$ is natural in finite Galois objects $A$.

For any atom $X$ in $\Ee$ with Galois covering $A\onto X$ we have a commutative square$$\xymatrix{x_\Ee^*(A)\ar[r]^\cong\ar@{->>}[d]&\tilde{x}^*_\Ee(A)\ar@{->>}[d]\\x_\Ee^*(X)\ar[r]_\cong&\tilde{x}^*_\Ee(X)}$$in which the lower bijection is uniquely determined by the upper bijection. It follows then from Lemma \ref{Galois0} that the fibre functors $x_\Ee^*$ and $\tilde{x}_\Ee^{*}$ are isomorphic on the full subcategory $\Ee_{at}$ spanned by atoms. Since Theorem \ref{main1} yields the identifications $s\Ee_{at}=\Ee_{sf}=\Ee$, the fibre functors are globally isomorphic.\end{proof}



We now aim at giving a site characterisation for connected, finitely generated Grothendieck toposes $\Ee$ along with an explicit construction of Galois points.\vspace{1ex}

The \emph{finite sum completion} $\Cc^\oplus$ of a category $\Cc$ has as objects pairs $(I,(A_i)_{i\in I})$ consisting of a finite set $I$ and a family of objects of $\Cc$ indexed by $I$. The morphisms $(I,A_i)\to (J,B_j)$ are pairs $(\phi,f_i)$ consisting of a mapping $\phi:I\to J$ and morphisms $f_i:A_i\to B_{\phi(i)}$ in $\Cc$, with the obvious composition law. If $\Cc$ already has finite sums, then there is a canonical functor $\Cc^\oplus\to\Cc:(I,A_i)\mapsto \bigoplus_{i\in I}A_i$. The category $\Cc^\oplus$ has finite sums so that we get a functor $\mu_\Cc:(\Cc^\oplus)^\oplus\to\Cc^\oplus$ as well as an inclusion functor $\eta_\Cc:\Cc\to\Cc^\oplus$. The triple $((-)^\oplus,\mu,\eta)$ constitutes a monad on the category of small categories whose algebras are precisely categories with finite sums. Johnstone uses a similar construction to ``positivize'' any coherent category, cf. \cite[Proposition A.1.4.5]{J}.

\begin{dfn}A \emph{Galois site} $\Aa$ is a strict atomic site\footnote{An atomic site is a site whose covering sieves are the non-empty sieves. The axioms for a Grothendieck topology are satisfied if and only if for any morphism cospan $\alpha:A\to C,\,\beta:B\to C$ there exists a morphism span $\alpha':D\to A,\,\beta':D\to B$ such that $\alpha\alpha'=\beta\beta'$. We call an atomic site \emph{strict} if all morphisms of the site are strict epimorphisms. This amounts to the condition that the Yoneda embedding of the site into its category of sheaves is full and faithful, cf. \cite[7(3)]{BD1}}in which for each object $X$ there exists an object $A_X$ such that the morphism-set $\Aa(A_X,X)$ is finite and the canonical span $A_X\leftarrow A_X^{\oplus\Aa(A_X,X)}\rightarrow X$ realises the product $A_X\times X$ in $\Aa^\oplus$.\end{dfn}

\begin{prp}\label{Galoissite}A connected Grothendieck topos is finitely generated if and only if it is equivalent to the category of sheaves on a Galois site with terminal object. The Galois site may be chosen such that its finite sum completion is a pretopos.\end{prp}

\begin{proof}We assume first that $\Ee$ is a connected, finitely generated Grothendieck topos. By Theorem \ref{main1} $\Ee$ is atomic so that by \cite[Theorem 1]{BD1} the full subcategory $\Aa$ of atoms of $\Ee$ is a strict atomic site inducing an equivalence $\Ee\simeq\Sh(\Aa)$. Since $\Ee$ is connected, $\Aa$ has a terminal object. Since all atoms are locally finite and connected we have a full inclusion $\Aa\hookrightarrow\Ee_f$. The induced functor $\Aa^\oplus\to\Ee_f$ admits a quasi-inverse taking an object of $\Ee_f$ to the finite family of its connected components. The quasi-inverse is functorial by extensivity of $\Ee$. So $\Aa^\oplus$ may be identified with $\Ee_f$ which is a pretopos according to Proposition \ref{pretopos}.

By Proposition \ref{finiteGalois}, each atom $A$ admits a finite Galois covering $A_X$. The cardinality of the morphism-set $\Aa(A_X,X)$ equals the cardinality of the quotient $\Aut(A_X)/\Aut_X(A_X)$ which is finite because $\Aut(A_X)$ is a subgroup of a finite permutation group. Since $A_X$ splits $X$, the product $A_X\times X$ is isomorphic to $A_X\times\gamma^*(\{1,\dots,n\})$ which in turn is isomorphic to a sum of $n$ copies of $A_X$. Under this identification, the projection to $A_X$ is the iterated codiagonal $A_X^{\oplus n}\to A_X$ while the projection to $X$ is the morphism $A_X^{\oplus n}\to X$ whose summands are given by the individual morphisms $A_X\to X$. So $\Aa$ is a Galois site with terminal object.

Conversely, assume that $\Aa$ is a Galois site with terminal object. Since $\Aa$ generates $\Sh(\Aa)$ it remains to be shown that the objects $X$ of $\Aa$ are finite objects of $\Sh(\Aa)$. Since they are connected, it is enough to show that they are locally finite. By hypothesis we have isomorphisms $A_X\times X\cong A_X^{\oplus\Aa(A_X,X)}\cong A_X\times\gamma^*(\Aa(A_X,X))$ for a finite morphism-set $\Aa(A_X,X)$. All objects $X$ of $\Aa$ are thus locally finite.\end{proof}

\begin{dfn}[cf. \cite{SGA1}]A \emph{Galois category} $(\Cc,F_\Cc)$ consists of a pretopos $\Cc$ with complemented subobjects and an exact conservative fibre functor $F_\Cc:\Cc\to\Ss_f$.\end{dfn}

\begin{prp}\label{exact}The embedded pretopos $\Ee_f$ of a connected, finitely generated Grothendieck topos $\Ee$ is a Galois category for a fibre functor $F_{\Ee_f}$ assigning to a each atom $\,X$ the set $\gamma_!(A_X\times X)$ of connected components of $A_X\times X$ where $A_X\to X$ is a fixed Galois covering of $X$.\end{prp}

\begin{proof}Since $\Ee_f$ is essentially small (cf. the proof of Theorem \ref{main1}) we can choose once and for all a Galois covering $A_X$ for each atom $X$ of $\Ee$. The assignment $X\mapsto \gamma_!(A_X\times X)$ defines a functor $\Ee_{at}\to\Ss_f$ because any two liftings $\bar{f}:A_X\to A_Y$ of a morphism of atoms $f:X\to Y$ differ only by an automorphism in $\Aut_Y(A_Y)$ (cf. the proof of Lemma \ref{Galois0}) so that the induced map $\gamma_!(\bar{f}\times f):\gamma_!(A_X\times X)\to\gamma_!(A_Y\times Y)$ does not depend on the choice of the lifting $\bar{f}$.

Like in the proof of Proposition \ref{Galoissite} we can apply finite sum completion to define a functor $F_{\Ee_f}:\Ee_f=\Ee_{at}^\oplus\to\Ss_f^\oplus\to\Ss_f$. It remains to be shown that $F_{\Ee_f}$ is exact and conservative. It is actually enough to show that $F_{\Ee_f}$ is exact and conservative when restricted to the full subcategory $\Spl(A)_f$ of $\Ee_f$ split by a finite Galois object $A$, because any finite set of objects of $\Ee_f$ is contained in such a subcategory.

We shall show that the functor $F_{\Ee_f}$ when restricted to $\Spl(A)_f$ is isomorphic to the composite functor $\Spl(A)_f\simeq\BB\Aut(A)_f\to\Ss_f$ where the first is an equivalence and the second is the exact and conservative forgetful functor, cf. Section \ref{Galoissection}.

Since $F_{\Ee_f}$ preserves finite sums, we can restrict to atoms in $\Spl(A)_f$. Lemma \ref{Galois0} and Proposition \ref{Galoissite} imply isomorphisms $\gamma_!(A_X\times X)\cong\Aut(A_X)/\Aut_X(A_X)$ natural in atoms $X$. By the universal property of the covering $A_X\onto X$, any epimorphism $A\onto X$ factors through $A\onto A_X$, uniquely determined up to an element in $\Aut_X(A_X)$. This induces an isomorphism $\gamma_!(A\times X)\to\gamma_!(A_X\times X)$ not depending on the choice of the epimorphism $A\onto A_X$. We get natural isomorphisms $\gamma_!(A\times X)\cong\Aut(A)/\Aut_X(A)$ which, via the equivalence $\Spl(A)_f\simeq\BB\Aut(A)_f$, induce the forgetful functor $\BB\Aut(A)_f\to\Ss_f$, cf. Section \ref{Galoissection}.\end{proof}

\begin{rmk}Proposition \ref{exact} is closely related to the Galois theory developed by Barr in \cite{Barr1, Barr2}. Our Galois coverings are dual to Barr's normal envelopes. Barr constructs in \cite{Barr2} a functor $M$ which is similar to our fibre functor $F_{\Ee_f}$. He defines $M$ by choosing varying envelopes in order to validate functoriality, exactness and conservativity of $M$. This leads to delicate coherence problems which we circumvent by Galois-theoretical methods after fixing a Galois covering for each atom of $\Ee$.\end{rmk}

\begin{dfn}\label{profinitegroup}A \emph{profinite group} is a group $\,G$ endowed with a filter of subgroups $(K_\alpha)_{\alpha\in\Ff}$ such that
\begin{itemize}\item[(a)]each subgroup $K_\alpha$ has finite index in $\,G$;\item[(b)]the filter is stable under conjugation in $\,G$;\item[(c)]the intersection $\bigcap_{\alpha\in\Ff}K_\alpha$ is trivial.\end{itemize}\end{dfn}

We endow $G$ with the unique topology turning $G$ into a topological group so that $(K_\alpha)_{\alpha\in\Ff}$ coincides with the filter of open subgroups of $G$. The orbit under conjugation of every $K_\alpha$ is finite because the normaliser of $K_\alpha$ has finite index. The finite intersection of the conjugates of $K_\alpha$ is an open normal subgroup $N_\alpha$ of $G$, contained in $K_\alpha$. The $N_\alpha$ form therefore a neighbourhood basis of the neutral element of $G$ so that $G$ may be written as a cofiltered limit of finite groups $G/N_\alpha$. By (c) the topology of $G$ satisfies Hausdorff's separation axiom so that $G$ is a compact zero-dimensional space, a so-called Stone space. Every group object in the category of Stone spaces arises in this way from its filter of open subgroups.

The \emph{classifying topos} $\BB G$ of a profinite group $G$ is the category of discrete sets with \emph{continuous} $G$-action or, what amounts to the same, the category of those $G$-sets whose isotropy groups belong to the given filter of open subgroups of $G$.

\begin{thm}\label{main2}Each connected, finitely generated Grothendieck topos $\Ee$ admits a Galois point $x_\Ee$ whose automorphism group $\Aut(x_\Ee)$ carries a unique profinite topology with the property that $\Ee$ is equivalent to the classifying topos $\BB\Aut(x_\Ee)$.\end{thm}

\begin{proof}By Proposition \ref{exact} there is an exact conservative functor $\Ee_f\to\Ss_f$ inducing a functor $\Ee=\Ee_{sf}\to\Ss_{sf}=\Ss$. Since small sums can be written as filtered colimits of finite sums, and the latter are exact in $\Ee$ and in $\Ss$, we get an exact conservative functor $\Ee\to\Ss$ preserving all small colimits. It is thus the inverse image functor of a surjective geometric morphism $x_\Ee:\Ss\to\Ee$.

The fibre $x_\Ee^*(X)$ at an atom $X$ is in bijection with the morphism-set $\Ee(A_X,X)$ where $A_X\onto X$ is a Galois covering. The automorphism group $\Aut(A_X)$ acts transitively on $\Ee(A_X,X)$, cf. Section \ref{Galoissection}. The morphism-set $\Ee(A_X,X)$ can thus be viewed as an atom of the classifying topos $\BB\Aut(A_X)$, and the composite geometric morphism $x^*_\Ee\circ (q_A)^*$ is isomorphic to the forgetful functor $(p_A)^*$. This holds for any atom of $\BB\Aut(A_X)$ and any finite Galois object $A_X$, i.e. $x_\Ee$ is a Galois point.

Let us describe the automorphism group $\Aut(x_\Ee)$. By Proposition \ref{Galoissite} $\Ee$ is equivalent to the topos of sheaves on the atomic Galois site $\Aa$ so that $\Aut(x_\Ee)$ is isomorphic to the automorphism group of the fibre functor restricted to $\Aa$.

The category of elements of the restricted fibre functor has as objects pairs $(X\in\Ob(\Aa),x\in x_\Ee^*(X))$, i.e. \emph{atomic neighbourhoods} of $x_\Ee$ (cf. \cite{SGA4}) and as morphisms $f:(X,x)\to(Y,y)$ those $f\in\Aa(X,Y)$ for which $x_\Ee^*(f)(x)=y$. A morphism of atomic neighbourhoods of $x_\Ee$ induces a commuting square$$\xymatrix{A_X\ar@{->>}[r]^{\bar{f}}\ar@{->>}[d]_x&A_Y\ar@{->>}[d]^y\\X\ar@{->>}[r]_f&Y}$$where $x:A_X\onto X$ and $y:A_Y\onto Y$ are induced by $x\in x_\Ee^*(X)$ and $y\in x_\Ee^*(Y)$ and the chosen Galois coverings of $X$ and $Y$. The lifting $\tilde{f}:A_X\to A_Y$ is determined up to an element in $\Aut_Y(A_Y)$, cf. Lemma \ref{Galois0}. We get a group epimorphism $\alpha_f:\Aut(A_X)\to\Aut(A_Y)$ such that $(\alpha_f)_!(x)=y$.

Since an automorphism of $x_\Ee$ corresponds to an invertible endofunctor $\alpha:\el(x_\Ee)\to\el(x_\Ee)$ over $\Aa$, i.e. to a compatible system of automorphisms $(\alpha_X\in\Aut_X(A_X))_{X\in\Ob(\Aa)}$, the indeterminacy of $\tilde{f}$ vanishes, and the automorphism group $\Aut(x_\Ee)$ may be identified with an actual limit of finite groups. This limit can be computed in the category of Stone spaces and carries thus a profinite topology such that $x_\Ee^*:\Ee=s\Aa\to s\Ss_f=\Ss$ takes values in \emph{continuous} $\Aut(x_\Ee)$-sets.

We now show that $x_\Ee^*:\Ee\to\Ss$ induces an equivalence $\Ee\simeq\BB\Aut(x_\Ee)$ by identifying the Galois site $\Aa$ with a \emph{site of definition} for $\BB\Aut(x_\Ee)$, cf. Mac Lane-Moerdijk \cite[Chapter III.9]{MM}. Such a site has as objects the transitive continuous $\Aut(x_\Ee)$-sets. These are up to isomorphism the transitive $\Aut(A_X)$-sets $M$ for finite Galois objects $A_X$ covering atoms $X$ of $\Ee$, cf. Section \ref{Galoissection}. The morphisms are given by equivariant maps $M\to N$ for some group epimorphism $\Aut(A_X)\to\Aut(A_Y)$. These are precisely the morphisms of the Galois site $\Aa$ by Lemma \ref{Galois0}.\end{proof}

\begin{rmk}\label{BhattScholze}Since for a profinite group $G$ the category of finite objects of the classifying topos $\BB G$ is the category of finite continuous $G$-sets, the equivalence $\Ee\simeq\BB\Aut(x_\Ee)$ recovers Grothendieck's Representation Theorem $\Ee_f\simeq\BB\Aut(x_\Ee)_f$ for the Galois category $(\Ee_f,F_{\Ee_f})$, cf. Proposition \ref{exact} and \cite[Theorem 4.1]{SGA1}.

It is well-known that a connected, atomic Grothendieck topos is equivalent to the classifying topos of the \emph{localic automorphism group} of any of its points, cf. \cite{JT, M, J, D, B}. Neither the points nor the associated localic automorphism groups need to be isomorphic. The distinguishing feature of Theorem \ref{main2} is the explicit construction, for any connected, finitely generated Grothendieck topos $\Ee$, of a Galois point with profinite automorphism group. The construction of this Galois point involves choices but any two Galois points are isomorphic by Proposition \ref{Galoispoint}.

Recently, there has been renewed interest in Grothendieck toposes $\Ee$ equivalent to $\BB G$ for a topological automorphism group of some surjective point, cf. Bhatt-Scholze \cite[Theorem 7.2.5]{BS} and Caramello \cite[Theorem 3.5]{C}. The occurring automorphism groups are \emph{not} assumed prodiscrete so that this \emph{topological} Galois theory goes beyond the classical context of Galois toposes. The paradigmatic example is the Myhill-Schanuel topos $\BB\Sgm_\infty$ of \emph{nominal sets} where the infinite symmetric group $\Sgm_\infty$ is endowed with the topology induced from the product topology on $\NN^\NN$. This topological group is \emph{coherent} in the sense of Johnstone \cite[Example D.3.4.1]{J} so that $\BB\Sgm_\infty$ is a coherent, atomic Grothendieck topos which is not finitely generated. There is a Representation Theorem for this kind of toposes, due to Makkai \cite{Ma0}.

In general, any point of a connected, atomic Grothendieck topos $\Ee$ is surjective so that the Representation Theorem of Butz-Moerdijk \cite{BM} yields the existence of a topological group $G$ such that $\Ee\simeq\BB G$. Up to Morita equivalence, the topology of $G$ can be replaced with the (possibly coarser) filter-topology induced by the open subgroups of $G$, and $G$ with this filter-topology becomes a \emph{complete topological group} with respect to its \emph{two-sided uniformity}, cf. \cite[Proposition 7.1.5]{BS}.

Interestingly, these complete topological groups have the property that their frames of opens are cogroups in frames so that they induce localic groups which are hypercomplete in Isbell's sense \cite{Isbell}, cf. K\v{r}\'{i}\v{z} \cite{Kriz} and Banaschewski-Vermeulen \cite{BV}.\end{rmk}

\begin{exm}\label{Speck}\emph{The \'etale topos of a field $k$.} For any field $k$, the opposite of the category of finite separable field extensions of $k$ is a \emph{Galois site} $\Aa_k$ in the sense of Definition \ref{Galoissite}. Indeed, by the \emph{Primitive Element Theorem}, any finite separable extension $k\hookrightarrow K$ is of the form $k\hookrightarrow k[x]$ for an element $x\in K$. A \emph{splitting field} $K'$ for the minimal monic polynomial $f\in k[X]$ defining $x\in K$ has then the property that $k[x]\otimes_kK'\cong K'[x]\cong K'^{\deg(f)}$ which is dual to the defining property of a Galois site. The finite sum completion $\Aa_k^\oplus$ is equivalent to the opposite of the category of \emph{finitely generated} separable $k$-algebras. The topos $\Et(k)$ of sheaves on $\Aa_k^\oplus$ is usually called the \emph{\'etale topos} of $\Spec(k)$. According to Theorem \ref{main2} this topos is equivalent to the classifyng topos $\BB\Aut(x_{\Et(k)})$ of the profinite automorphism group of a Galois point $x_{\Et(k)}$. This profinite group may be identified with the \emph{absolute Galois group} $\Gal(\bar{k}/k)$ of $k$, since the fibre functor $x^*_{\Et(k)}$ at $\Spec(K)$ is isomorphic to the set of field embeddings $K\hookrightarrow \bar{k}$ into a separable closure $\bar{k}$ of $k$. The resulting profinite topology on $\Gal(\bar{k}/k)$ is known as the \emph{Krull topology}.\end{exm}

\section{Profinite fundamental group}

Combining Theorems \ref{main1} and \ref{main2} we obtain for every connected Grothendieck topos $\Ee$ a factorisation of the global section functor into $\Ee\to\Ee_{sf}\to\Ss$ where the intermediate topos $\Ee_{sf}$ is finitely generated and thus admits a Galois point.

We say that $\Ee$ is \emph{Galois-pointed} if $\Ee$ comes equipped with a point $x:\Ss\to\Ee$ such that the quotient $\bar{x}:\Ss\to\Ee\to\Ee_{sf}$ is a Galois point of $\Ee_{sf}$. This amounts to the property that the corresponding fibre functor $x^*:\Ee\to\Ss$ when restricted to finite objects can be computed by Galois-theoretic means like in Theorem \ref{main2}. We define the \emph{profinite fundamental group} $\hpi_1(\Ee,x)$ of a Galois-pointed connected Grothendieck topos $(\Ee,x:\Ss\to\Ee)$ to be the profinite automorphism group $\Aut(\bar{x}^*)$ of the restricted fibre functor $\bar{x}^*=x^*|_{\Ee_{sf}}$.

The aim of this section is to explore functoriality properties of this profinite fundamental group construction. In order to establish mere functoriality of $\hpi_1$ we must first show that a pointed geometric morphism $\phi:(\Ee,x)\to(\Ff,y)$ induces a pointed geometric morphism $\phi_{sf}:(\Ee_{sf},\bar{x})\to(\Ff_{sf},\bar{y})$ between the finitely generated quotient toposes. For this it is enough to show that the inverse image functor $\phi^*:\Ff\to\Ee$ takes $\Ff_{sf}$ to $\Ee_{sf}$. Since inverse image functors preserve locally finite objects and arbitrary sums, this follows from the following lemma.

\begin{lma}\label{locallyfinite}In any pointed connected Grothendieck topos, locally finite objects are decomposition-finite.\end{lma}

\begin{proof}Let $x:\Ss\to\Ee$ be a point and $X$ be a locally finite object of $\Ee$. Since $\Ee$ is connected either $X$ is initial or $X$ is globally supported. In the latter case, either $X$ is connected and hence decomposition-finite, or $X$ is a sum $X_1+X_2$ of two non-initial locally finite objects of $\Ee$ by Lemma \ref{complemented3}. Since the fibre functor $x^*$ preserves locally finite objects, binary sums and globally supported objects, the finite non-empty set $x^*(X)$ breaks into two finite non-empty sets $x^*(X_1)+x^*(X_2)$. An induction on the cardinality of $x^*(X)$ yields decomposition-finiteness of $X$.\end{proof}

A geometric morphism $\phi:\Ee\to\Ff$ is said to be \emph{(finite-)\'etale} if $\phi$ is equivalent over $\Ff$ to a geometric morphism $\Ff/F\to\Ff$ for a (finite) object $F$ of $\Ff$.

Recall that any \emph{locally connected} geometric morphism factorises as a \emph{connected}, locally connected geometric morphism $\Ee\to\Ff/\phi_!(1_\Ee)$ followed by an \emph{\'etale} geometric morphism $\Ff/\phi_!(1_\Ee)\to\Ff$, cf. Johnstone \cite{J1}.

\begin{prp}\label{proepimono}A pointed geometric morphism $\phi:(\Ee,x)\to(\Ff,y)$ between Galois-pointed connected Grothendieck toposes induces a homomorphism of profinite fundamental groups $\hpi_1(\phi):\hpi_1(\Ee,x)\to\hpi_1(\Ff,y)$ which is surjective (resp. injective) whenever $\phi$ is connected (resp. finite-\'etale).

In particular, if $\phi$ is locally connected of finite type, then its connected/\'etale factorisation induces the epi/mono factorisation on profinite fundamental groups.
\end{prp}

\begin{proof}By Lemma \ref{locallyfinite} we have a commutative diagram of geometric morphisms$$\xymatrix{\Ss\ar[r]^{x}\ar[rd]_{\bar{x}}&\Ee\ar[r]^\phi\ar[d]&\Ff\ar[d]\\&\Ee_{sf}\ar[r]_{\phi_{sf}}&\Ff_{sf}}$$such that $\phi_{sf}$ takes the Galois point $\bar{x}$ to the Galois point $\bar{y}$ and $\phi^*:\Ff\to\Ee$ restricts to $\phi_{sf}^*:\Ff_{sf}\to\Ee_{sf}$. The group homomorphism $\hpi_1(\phi):\hpi_1(\Ee,x)\to\hpi_1(\Ff,y)$ may be computed by precomposing the elements of $\Aut(\bar{x}^*)$ with $\phi_{sf}^*$.

If $\phi:\Ee\to\Ff$ is connected, then $\phi_{sf}:\Ee_{sf}\to\Ff_{sf}$ is also connected because the vertical geometric morphisms of the diagram above are connected. Therefore $\phi_{sf}^*$ is (up to an invertible $2$-cell) a section, so that precomposing with $\phi_{sf}^*$ is surjective.

If $\phi:\Ee\to\Ff$ is finite-\'etale with respect to an object $F$ of $\Ff_f$, then $\phi_{sf}:\Ee_{sf}\to\Ff_{sf}$ is also \'etale with respect to $F$. The group $\hpi_1(\Ee,x)=\Aut(\bar{x}^*)$ may then be identified with the subgroup of $\hpi_1(\Ff,y)=\Aut(\bar{y}^*)$ consisting of those automorphisms which fix $F$. We use here that liftings of $\bar{y}:\Ss\to\Ff_{sf}$ along $\Ff_{sf}/F\to\Ff_{sf}$ correspond functorially to elements of $\bar{y}^*(F)$, cf. Grothendieck-Verdier \cite[Proposition IV.5.12]{SGA4}.

For the last assertion it suffices now to observe that composing a Galois point $\bar{x}:\Ss\to\Ee_{sf}$ with a connected, locally connected morphism $\phi_{sf}:\Ee_{sf}\to\Ee'_{sf}$ yields a Galois point of $\Ee'_{sf}$ because the functor $\phi_{sf}^*$ is fully faithful and preserves connected objects (and hence Galois objects) by Lemma \ref{locallyconnected2}.
\end{proof}

We shall now compare the fundamental group $\pi_1(E,x)$ of a based topological space $E$ with the profinite fundamental group $\hpi_1(\Ee,x)$ of the topos $\Ee=\Sh(E)$ of set-valued sheaves on $E$, where the geometric point $x:\Ss=\Sh(\star)\to\Ee=\Sh(E)$ is induced by the corresponding base point $x:\star\to E$. This is more intricate than one might expect because Grothendieck's Galois-theoretic approach to fundamental groups has no simple-minded analog in the setting of topological spaces, points of localic toposes admitting no non-trivial automorphisms.

What replaces Galois points of a finitely generated Grothendieck topos are the topologist's \emph{universal covering spaces}, and any comparison unavoidably involves a topos-theoretic understanding of universal covering spaces. This is the route we are choosing, closely following ideas going back to Chevalley \cite[Chapter II]{Ch}. The reader may consult Barr-Diaconescu \cite[Examples]{BD2}, Bunge-Moerdijk \cite[Sections 1-2]{BMoe} and Borceux-Janelidze \cite[Chapter 6]{BJ} for more background on the matter.

\begin{dfn}A geometric morphism $\phi:\Ee\to\Ff$ between pointed, connected Grothendieck toposes is said to be

\begin{itemize}\item\emph{Galois-connected} if a locally constant object $\,Y$ of $\,\Ff$ is constant as soon as its inverse image $\phi^*(Y)$ is constant in $\,\Ee$;\item \emph{Galois-\'etale} if it is \'etale with respect to a Galois object $\,U$ of $\,\Ff$.\end{itemize}\end{dfn}

Any Galois object is globally supported and split by itself and hence in particular a locally constant object. Therefore, any Galois-\'etale morphism is a \emph{covering projection}. The group of deck transformation of this covering projection may be identified with the automorphism group of the corresponding Galois object $U$.

The following terminology is inspired by the treatment of \emph{comprehensive factorisation systems} of the first author and Kaufmann in \cite[Section 3]{BK}. It would be interesting to know which geometric morphisms allow for a comprehensive factorisation in the sense of Definition \ref{comprehensive0}. Caramello-Osmond \cite{CO} show that \emph{any} geometric morphism allows for a terminally-connected/pro\'etale factorisation which coincides with the connected/\'etale factorisation for \emph{locally connected} geometric morphisms.

In general one has strict implications connected$\implies$terminally-connected$\implies$Galois-connected as well as Galois-\'etale$\implies$\'etale$\implies$pro\'etale.


\begin{dfn}\label{comprehensive0}A \emph{point} $x:\Ss\to\Ee$ is called \emph{comprehensive} if it factors as a Galois-connected morphisms $\Ss\to\bar{\Ee}$ followed by a Galois-\'etale morphism $\bar{\Ee}\to\Ee$.

The \emph{Chevalley fundamental group} $\pi_1(\Ee,x)$ at a comprehensive point $x:\Ss\to\Ee$ is the group of deck transformations of such a Galois-\'etale morphism $\bar{\Ee}\to\Ee$.\end{dfn}

\begin{lma}\label{comprehensive}Let $x:\Ss\to\Ee$ be a point of a connected Grothendieck topos.

\begin{itemize}\item[(a)]Any two Galois-connected/Galois-\'etale factorisations of $x$ are equivalent and yield thus isomorphic Chevalley fundamental groups.\item[(b)]For a path-connected, semi-locally simply connected topological space $E$, any point $x:\Ss\to\Sh(E)$ is comprehensive, and the Chevalley fundamental group $\pi_1(\Sh(E),x)$ is isomorphic to the usual fundamental group $\pi_1(E,x)$.\item[(c)]If $x$ is comprehensive then the induced point $\bar{x}:\Ss\to\Ee_{sf}$ is a Galois point.\end{itemize}\end{lma}

\begin{proof}(a) The following commutative square of undotted arrows$$\xymatrix{\Ss\ar[r]^{\iota_1}\ar[d]_{\iota_2}&\bar{\Ee}_1\ar[d]^{\rho_1}\\\bar{\Ee}_2\ar[r]_{\rho_2}\ar@{.>}[ru]&\Ee}$$is defined by two Galois-connected/Galois-\'etale factorisations of a comprehensive point $x:\Ss\to\Ee$. Since in $\Ss$ all objects are constant, in $\bar{\Ee_1}$ and in $\bar{\Ee}_2$ all locally constant objects are constant. This is the topos-theoretic formulation of simple connectedness of $\bar{\Ee}_1$ and of $\bar{\Ee}_2$.

Let $U_1$ be a Galois object of $\Ee_1$ such that $\bar{\Ee_1}\simeq\Ee/U_1$. Then $\rho_2^*(U_1)$ is constant in $\bar{\Ee_2}$ so that the morphism $\star_\Ss\to\iota_2^*\rho_2^*(U_1)$ inducing the commuting square above (cf. \cite[Proposition IV.5.12]{SGA4}) is the image under $\iota_2^*$ of a morphism $\star_{\bar{\Ee_2}}\to\rho_2^*(U_1)$. The latter induces a diagonal filler $\bar{\Ee}_2\to\bar{\Ee}_1$ rendering the lower triangle commutative and the upper triangle pseudo-commutative by conservativity of $\rho_1^*$.

Similarily, choosing a Galois object $U_2$ such that $\bar{\Ee_2}\simeq\Ee/U_2$, there results a diagonal filler $\bar{\Ee}_1\to\bar{\Ee}_2$ quasi-inverse to the former diagonal filler. It follows that the Chevalley fundamental groups of the two factorisations are isomorphic.

(b) The hypotheses on the topological space $E$ amount to the existence of a universal covering space $\bar{E}(x)\to E$ for each base point $x\in E$ inducing a factorisation of $x:\Ss\to\Sh(E)$ into a Galois-connected morphism $\Ss\to\Sh(\bar{E}(x))$ followed by a Galois-\'etale morphism $\Sh(\bar{E}(x))\to\Sh(E)$. The Chevalley fundamental group of the latter is isomorphic to the group of deck transformations of the universal covering $\bar{E}(x)\to E$ which in turn is canonically isomorphic to $\pi_1(E,x)$.

(c) Let $U$ be the Galois object corresponding to the comprehensive point $x:\Ss\to\Ee$ and let $A$ be any finite Galois object of $\Ee$. We have to show that in the following commutative diagram of geometric morphisms$$\xymatrix{\Ss\ar@/^/[rrrd]_{x_A}\ar@/^3ex/[rr]^x\ar[r]^{\iota}\ar[rd]_{\bar{\iota}}&\Ee/U\ar[r]^\rho\ar[d]&\Ee\ar[d]\ar[rd]&\\
&(\Ee/U)_{sf}\ar[r]_{\rho_{sf}}&\Ee_{sf}\ar[r]_{q_A}&\BB\Aut(A)}$$
the point $x_A:\Ss\to\BB\Aut(A)$ is ``canonical'', i.e. an essential surjection, cf. Section \ref{Galoissection}. Indeed, since $\rho:\Ee/U\to\Ee$ is an essential surjection, the induced morphism $\rho_{sf}:(\Ee/U)_{sf}\to\Ee_{sf}$ is also an essential surjection. Since $\Ee/U$ is simply connected (every locally constant object is constant), the topos $(\Ee/U)_{sf}$ contains only constant objects so that $\bar{\iota}:\Ss\to(\Ee/U)_{sf}$ is an equivalence. Therefore, $x_A$ is the composite of three essential surjections and hence itself an essential surjection.\end{proof}

\begin{prp}\label{profinitecompletion}For any comprehensive point $x:\Ss\to\Ee$ of a connected Grothen-dieck topos, the profinite fundamental group $\hpi_1(\Ee,x)$ is isomorphic to the profinite completion $\widehat{\pi_1(\Ee,x)}$ of the Chevalley fundamental group $\pi_1(\Ee,x)$.\end{prp}

\begin{proof}Let $U$ be a Galois object corresponding to the comprehensive point $x$. The Chevalley fundamental group $\pi_1(\Ee,x)$ equals then the automorphism group $\Aut_\Ee(U)$. Finite quotients of $U$ correspond to finite Galois-objects $A$ which in turn correspond to finite topos quotients $q_A:\Ee_{sf}\to\BB\Aut(A)$. According to Lemma \ref{comprehensive}c, the quotient points are canonical so that the profinite completion of the group $\Aut_\Ee(U)$ is computed in the same way as the profinite automorphism group of the Galois point $\bar{x}:\Ss\to\Ee_{sf}$, cf. the proof of Theorem \ref{main2}.\end{proof}

\begin{exm}The fundamental group of the topological circle $S^1$ is $\ZZ$. It is isomorphic to the group of deck transformations of the universal covering $\RR\to S^1:t\mapsto e^{it}$. Therefore, the profinite fundamental group of $\Sh(S^1)$ is the profinite completion of $\ZZ$, i.e. a product of rings of $p$-adic integers where $p$ runs through all prime numbers.\end{exm}

\begin{rmk}The profinite fundamental group $\hpi_1(\Sh(E),x)$ can thus be viewed as a profinite generalisation of the classical fundamental group $\pi_1(E,x)$, and as such it is an algebraic invariant for all connected (not only the path-connected) topological spaces. The price to pay for this generalisation is that the passage from a group to its profinite completion loses information.

However, the motivation of Grothendieck \cite{SGA1} to introduce profinite fundamental groups comes from algebraic geometry where the underlying topology of a \emph{scheme} is not suited for comprehensive points, yet there is a well defined \emph{\'etale topos} for any scheme. The latter is finitely generated and thus admits a Galois point as well as a profinite fundamental group containing valuable information about the scheme.\end{rmk}

We end this article by constructing a topos-theoretical analog of Grothendieck's short exact sequence of profinite fundamental groups, cf. \cite[Theorem IX.6.1]{SGA1}.


\begin{thm}\label{fundamental}Let $\phi:(\Ee,x)\to(\Ff,y)$ be a pointed, connected geometric morphism between Galois-pointed, connected Grothendieck toposes.

If $\,\Ff$ is finitely generated then there exists a locally finite-\'etale morphism $\rho:\bar{\Ff}\to\Ff$ with simply connected domain $\bar{\Ff}$, and a pseudo-pullback square of toposes$$\xymatrix{\Ee\times_\Ff\bar{\Ff}\ar[d]\ar[r]&\bar{\Ff}\ar[d]^\rho\\\Ee\ar[r]_\phi&\Ff}$$inducing a short exact sequence of profinite fundamental groups$$\xymatrix{1\ar[r]&\hpi_1(\Ee\times_\Ff\bar{\Ff},z)\ar[r]&\hpi_1(\Ee,x)\ar[r]&\hpi_1(\Ff,y)\ar[r]&1}$$\end{thm}

\begin{proof}Since $\Ff$ is finitely generated, it may be identified with the topos of sheaves on a Galois site $\Aa$, cf. Proposition \ref{Galoissite}. The restricted fibre functor $\bar{y}^*:\Aa\to\Ss_f$ takes values in finite sets and surjections of finite sets. The category of elements $\bar{\Aa}=\el(\bar{y}^*)$ is then a strict atomic site because the projection $\bar{\Aa}\to\Aa$ is faithful. Since the projection is \emph{cover-reflecting} we get by \cite[Proposition C.2.3.18]{J} a geometric morphism $\rho:\bar{\Ff}=\Sh(\bar{\Aa})\to\Ff=\Sh(\Aa)$ whose restriction to each atom $A$ induces a finite-\'etale morphism $\bar{\Ff}/\rho^*(A)\to\Ff/A$. The inverse images $\rho^*(A)$ are finite objects of $\bar{\Ff}$ and generate $\bar{\Ff}$ so that $\bar{\Ff}$ is finitely generated. By construction, the fibre functor for $\bar{\Aa}$ takes atoms to singleton sets. It follows that the profinite group classifying $\bar{\Ff}$ must be trivial, i.e. $\bar{\Ff}$ is simply connected.

So, the pseudo-pullback above exists and exhibits $\Ee\times_\Ff\bar{\Ff}\to\Ee$ as a locally finite-\'etale geometric morphism. Since $\phi$ is pointed, the Galois point $x:\Ss\to\Ee$ factors through a point $z:\Ss\to\Ee\times_\Ff\bar{\Ff}$. Even though it is unclear whether $z$ is a Galois-point, its automorphism group is a subgroup of the profinite automorphism group $\hpi_1(\Ee,x)$ and carries as such a canonical profinite topology induced by the fact that $\Ee\times_\Ff\bar{\Ff}\to\Ee$ is locally finite-\'etale. By Proposition \ref{proepimono}, we know that $\hpi_1(\phi):\hpi_1(\Ee,x)\to\hpi_1(\Ff,y)$ is an epimorphism. Every element of $\Aut(\bar{x}^*)=\hpi_1(\Ee,x)$ whose image under precomposition with $\phi^*$ is trivial in $\Aut(\bar{y}^*)=\hpi_1(\Ff,y)$ lifts consistently to the identity of the fibre functor of $\bar{\Ff}$, and comes therefore from an automorphism in the subgroup $\Aut(\bar{z}^*)=\hpi_1(\Ee\times_{\Ff}\bar{\Ff},z)$ of $\hpi_1(\Ee,x)$.\end{proof}

\begin{exm}\emph{The \'etale topos of a $k$-scheme.} For any scheme $X$ over a field $k$ Grothendieck and his school \cite{SGA4} define an \'etale topos $\Et(X)$ whose objects are \'etale maps of finite type with codomain $X$. For a punctual scheme $X=\Spec(k)$ this recovers the \'etale topos $\Et(k)$ discussed in Example \ref{Speck}. For a geometrically connected $k$-scheme $X$ the induced morphism $\Et(X)\to\Et(k)$ is connected. The fibre topos constructed in Theorem \ref{fundamental} is then easily identified with the \'etale topos $\Et(X\otimes_k\bar{k})$. The short exact sequence of \'etale fundamental groups becomes$$\xymatrix{1\ar[r]&\hpi_1(X\otimes_k\bar{k},z)\ar[r]&\hpi_1(X,x)\ar[r]&\Gal(\bar{k}/k)\ar[r]&1}$$ In the special case $k=\QQ$ and $X=\PP^1_\QQ-\{0,1,\infty\}$ one shows (using among others Proposition \ref{profinitecompletion}) that the \'etale fundamental group of $X\otimes_\QQ\bar{\QQ}$ is the profinite completion $\widehat{F}_2$ of the free group $F_2$ on two generators. The short exact sequence induces then an action (by outer automorphisms) of the absolute Galois group $\Gal(\bar{\QQ}/\QQ)$ on $\widehat{F}_2$. This action is closely related to ongoing research on profinite Grothendieck-Teichm\"uller groups and mapping class groups, cf. Schneps \cite{S}.\end{exm}

\vspace{5ex}

\noindent{\small\sc Universit\'e C\^ote d'Azur, Lab. J. A. Dieudonn\'e, Parc Valrose, 06108 Nice Cedex, France.}\hspace{2em}\emph{E-mails:}
clemens.berger and victor.iwaniack$@$univ-cotedazur.fr\vspace{1ex}


\begin{thebibliography}{99}

\bibitem{AL}O. Acu\~{n}a-Ortega and F. E. J. Linton -- \emph{Finiteness and Decidability I,} Lect. Notes Math. \textbf{753} (1979), 80--100.

\bibitem{BV}B. Banaschewski and J. J. C. Vermeulen -- \emph{On the completeness of localic groups}, Comment. Math. Univ. Carolin. \textbf{40} (1999), 293--307.

\bibitem{Barr0}M. Barr -- \emph{Toposes without points}, J. Pure Appl. Alg. \textbf{5} (1974), 265--280.

\bibitem{Barr1}M. Barr -- \emph{Abstract Galois Theory}, J. Pure Appl. Alg. \textbf{19} (1980), 21--42.

\bibitem{Barr2}M. Barr -- \emph{Abstract Galois Theory II}, J. Pure Appl. Alg. \textbf{25} (1982), 227--247.

\bibitem{BD1}M. Barr and R. Diaconescu -- \emph{Atomic toposes},  J. Pure Appl. Alg. \textbf{17} (1980), 1--24.

\bibitem{BP}M. Barr and R. Par\'e -- \emph{Molecular toposes}, J. Pure Appl. Alg. \textbf{17} (1980), 127--152.

\bibitem{BD2}M. Barr and R. Diaconescu -- \emph{On locally simply connected toposes and their fundamental groups}, Cahiers Top. G\'eom. Diff. \textbf{22} (1981), 301--317.

\bibitem{BK}C. Berger and R. M. Kaufmann -- \emph{Comprehensive factorisation systems}, Tbilisi Math. J. \textbf{10}(3) (2017), 255--277.

\bibitem{BS}B. Bhatt and P. Scholze -- \emph{The pro-\'etale topology for schemes}, Ast\'er. \textbf{369} (2015), 99--201.

\bibitem{BJ}F. Borceux and G. Janelidze -- \emph{Galois Theories}, Camb. Studies in Adv. Math. \textbf{72}, Cambridge University Press, 2001.

\bibitem{B}M. Bunge --- \emph{Galois groupoids and covering morphisms in topos theory}, AMS Fields Inst. Commun. \textbf{43} (2004), 131--161.

\bibitem{BF}M. Bunge and J. Funk -- \emph{Quasicomponents in topos theory: the hyperpure, complete spread factorization}, Math. Proc. Camb. Philos. Soc. \textbf{142} (2007), 47--62.

\bibitem{BMoe}M. Bunge and I. Moerdijk -- \emph{On the construction of the Grothendieck fundamental group of a topos by paths}, J. Pure Appl. Alg. \textbf{116} (1997), 99--113.

\bibitem{BM}C. Butz and I. Moerdijk -- \emph{Representing topoi by topological groupoids}, J. Pure Appl. Alg. \textbf{130} (1998), 223--235.

\bibitem{C}O. Caramello -- \emph{Topological Galois theory}, Adv. Math. \textbf{291} (2016), 646--695.

\bibitem{CO}O. Caramello and A. Osmond -- \emph{On a (terminally connected, pro-etale) factorization of geometric morphisms}, arXiv:2502.04213, 39 pp.

\bibitem{Ch}C. Chevalley -- \emph{Theory of Lie groups, I.} Princeton Math. Ser., \textbf{8} (1946), ix+217 pp.

\bibitem{D}E. Dubuc --- \emph{Localic Galois theory}, Adv. Math. \textbf{175} (2003), 144--167.

\bibitem{SGA1}A. Grothendieck --  \emph{SGA 1} (1960-1961), expos\'e V (Le groupe fondamental. G\'en\'eralit\'es), Lect. Notes Math. \textbf{224} (1971), 98--114.

\bibitem{SGA4}A. Grothendieck and J.-L. Verdier -- \emph{SGA 4} (1963-1964), expos\'e IV (Topos) and expos\'e VI (Conditions de finitude. Topos fibr\'es), Lect. Notes Math. \textbf{270} (1972), 154--271 and 377--467.

\bibitem{H}S. Henry -- \emph{On toposes generated by cardinal finite objects,} Math. Proc. Camb. Philos. Soc. \textbf{165} (2018), 209--223.

\bibitem{Isbell}J. R. Isbell -- \emph{Atomless parts of spaces}, Math. Scand. \textbf{31} (1972), 5--32.

\bibitem{J-1}P. T. Johnstone -- \emph{Open maps of toposes,} manuscripta math. \textbf{31} (1980), 217--247.

\bibitem{J0}P. T. Johnstone -- \emph{Factorization theorems for geometric morphisms I,} Cahiers Top. G\'eom. Diff. \textbf{22} (1981), 3--17.

\bibitem{J1}P. T. Johnstone -- \emph{Factorization theorems for geometric morphisms II,} Lect. Notes Math. \textbf{915} (1982), 216--233.

\bibitem{JStone}P. T. Johnstone -- \emph{Stone spaces}, Camb. Studies in Adv. Math. \textbf{3} (1982), xxi+370 pp.

\bibitem{J}P. T. Johnstone -- \emph{Sketches of an Elephant - A Topos Theory Compendium I/II,} Oxford Logic Guides \textbf{44} (2002), xxii+1089 pp.

\bibitem{JT}A. Joyal and M. Tierney -- \emph{An extension of the Galois theory of Grothendieck,} Mem. Amer. Math. Soc. \textbf{51} (1984), vii+71 pp.

\bibitem{KLM}A. Kock, P. Lecouturier and C. J. Mikkelsen -- \emph{Some topos theoretic concepts of finiteness,} Lecture Notes Math. \textbf{445} (1975), 209--283.

\bibitem{Kriz}I. K\u{r}\'{i}\u{z} -- \emph{A direct description of uniform completion in locales and a characterization of $LT$-groups}, Cahiers Top. et G\'eom. Diff. Cat\'eg. \textbf{27} (1986), 19--34.

\bibitem{L}O. Leroy -- \emph{Groupo\"\i de fondamental et th\'eor\`eme de van Kampen en th\'eorie des topos}, Cah. math. Montpellier (1979). https://agrothendieck.github.io/divers/leroy.pdf

\bibitem{Lurie}J. Lurie -- \emph{Ultracategories}, https://people.math.harvard.edu/~lurie/papers/Conceptual.pdf.

\bibitem{MM}S. Mac Lane and I. Moerdijk -- \emph{Sheaves in geometry and logic}, Universitext, Springer (1994).

\bibitem{Ma0}M. Makkai -- \emph{Full continuous embeddings of toposes,} Trans. Amer. Math. Soc. \textbf{269} (1982), 167--196.

\bibitem{Ma}M. Makkai -- \emph{Ultraproducts and categorical logic}, Lect. Notes Math. \textbf{1130} (1985), 222--309.

\bibitem{M}I. Moerdijk -- \emph{Prodiscrete groups and Galois toposes}, Nederl. Akad. Wetensch. Indag. Math. \textbf{51} (1989), 219--234.

\bibitem{MV}I. Moerdijk and J.J.C. Vermeulen -- \emph{Proper maps of toposes}, Mem. Amer. Math. Soc. \textbf{148} (2000), x + 108 pp.

\bibitem{MW}I. Moerdijk and G.C. Wraigth -- \emph{Connected, locally connected toposes are path-connected}, Trans. Amer. Math. Soc. \textbf{295} (1986), 849--859.

\bibitem{S}L. Schneps -- \emph{The Grothendieck-Teichm\"uller group $\widehat{GT}$: a survey.} Geometric Galois actions 1,
London Math. Soc. Lecture Note Ser., \textbf{242} (1997), 183--203.

\end{thebibliography}
\end{document}